%% file: main.tex
\pdfoutput=1
\documentclass[11pt]{article}
\usepackage[a4paper,margin=2.5cm]{geometry}
\usepackage[T1]{fontenc}
\usepackage{amsmath,amssymb,amsthm}
\usepackage{newtxtext}
\usepackage{newtxmath}
\usepackage{bm}
\usepackage{graphicx}
\usepackage{url}
\usepackage{microtype}

\usepackage{natbib}
\bibpunct[, ]{(}{)}{,}{a}{}{,}

\usepackage[font=small,labelfont=bf]{caption}
\usepackage{subcaption}
\usepackage{booktabs}
\usepackage{multirow}
\usepackage{accents}
\usepackage[hidelinks]{hyperref}
\hypersetup{pdftitle={Efficient LP warmstart for linear modifications of the constraint matrix},
    pdfauthor={Guillaume Derval, Bardhyl Miftari, Damien Ernst, Quentin Louveaux}}

\theoremstyle{plain}
\newtheorem{theorem}{Theorem}
\newtheorem{lemma}{Lemma}
\newtheorem{proposition}{Proposition}
\newtheorem{corollary}{Corollary}

\makeatletter
\renewenvironment{proof}[1]{\par\normalfont\topsep6\p@\@plus6\p@\relax
    \trivlist\item[\hskip\labelsep\itshape #1\@addpunct{.}]\ignorespaces}{\endtrivlist\@endpefalse}
\newcommand{\Halmos}{\qed}
\long\def\arxiv@empty{}
\newsavebox{\arxiv@tablebox}
\long\def\TABLE#1#2#3{%
    \sbox{\arxiv@tablebox}{#2}%
    \centering
    \begin{minipage}{\wd\arxiv@tablebox}
        \caption{#1}%
        \usebox{\arxiv@tablebox}%
        \long\def\arxiv@note{#3}%
        \ifx\arxiv@note\arxiv@empty\else\par\smallskip{\footnotesize #3\par}\fi
    \end{minipage}}
\long\def\FIGURE#1#2#3{%
    \centering #1\par
    \caption{#2}%
    \long\def\arxiv@note{#3}%
    \ifx\arxiv@note\arxiv@empty\else{\footnotesize #3\par}\fi}
\makeatother
\newenvironment{APPENDICES}{\appendix}{}

\input{tables/numbers.tex}

\newcommand{\colR}[2]{\phantom{#1}\llap{$\displaystyle #2$}}
\newcommand{\colL}[2]{\rlap{$\displaystyle #2$}\phantom{#1}}

\title{Efficient LP warmstart for linear modifications of the constraint matrix}
\author{Guillaume Derval, Bardhyl Miftari, Damien Ernst, Quentin Louveaux\\[1ex]
    {\normalsize Montefiore Institute, University of Li\`ege, Li\`ege, Belgium}\\[0.5ex]
    {\small\texttt{gderval@uliege.be}, \texttt{bmiftari@uliege.be}, \texttt{dernst@uliege.be},
    \texttt{q.louveaux@uliege.be}}}
\date{}

\begin{document}
\maketitle

\begin{abstract}
We consider the problem of computing the optimal solution and objective value of a linear program whose constraint matrix changes linearly with a parameter, for a predefined set of values of this parameter. Based on the information given by a precomputed basis, we present three efficient LP warm-starting algorithms, based on the eigenvalue decomposition, the Schur decomposition and the Hessenberg decomposition, to evaluate the optimal solution and optimal objective value of these problems.
The three algorithms have an overall complexity of $O(pm^2+pmn)$ where $m$ (resp. $n$) is the number of constraints (resp. variables) of the original problem and $p$ the number of parameter values after an initial preprocessing step. We also provide theorems related to the optimality conditions to verify when a basis is still optimal.
When the perturbation of the basis matrix has a low rank $\rho$, as is often the case in practice, we then reduce all the computations to a $\rho\times \rho$ matrix built with $\rho$ solves with the factorization of the basis maintained by the simplex solver; with it, we compute, in each direction, an interval on which the basis is guaranteed to remain optimal, together with a bound on the objective. In practice, these intervals are close to the exact ones, and they remain useful for degenerate bases. We also propose a matrix-free method based on shifted Krylov subspaces. Numerical experiments on \nInst{} instances built from real models, with up to \mMax{} constraints, assess the accuracy, the cost and the range of applicability of all the methods. On the interval on which a basis remains optimal, once this interval is known, the low-rank and Krylov methods compute the optimal objective value \localNocertSpeedupLrHess{} to \localNocertSpeedupKrylov{} faster than a warm start of the solver, in median.


\smallskip\noindent\textbf{Funding:} This work was partially funded by the SCK-CEN SMR Chair Program.
\end{abstract}

\section{Introduction}\label{sec:intro}

Many real-life linear programs do not have a one-value fit for some of their coefficients. Typically, several values for these coefficients may be valid depending on the setting studied and hypotheses made. Sometimes, these programs evolve over time, which can also lead to changes in the coefficients. Assessing the behavior of both the optimal objective function and the optimal solution in relation to these changes is particularly important. This issue has been addressed by the fields of sensitivity analysis and parametric programming.

An efficient method for dealing with varying coefficients, particularly when they are in the objective function or in the right-hand side, is warmstarting. Indeed, in these cases, either primal or dual feasibility is preserved, which allows us to run a few iterations of either the primal or dual simplex.
However, when dealing with varying constraint matrix coefficients, the basis matrix changes, which imposes a heavier computational burden.
In this paper, we propose warm-starting algorithms to evaluate the optimal objective function and optimal solution in an exact manner for LPs whose constraint coefficients vary linearly.

We study problems of the type 
\begin{equation}
\begin{aligned}
    \mathcal{P}(\lambda)\equiv\min &\  c^{\top} x\\
    \text{s.t.} &\  Ax+\lambda Dx = b\\
    & \ x \ge 0 
\end{aligned}
\end{equation}
where $\lambda$ controls the linear change in coefficients and belongs to a finite discrete set $\Lambda$, and the uncertainty matrix $D$ gives the direction of this change. More precisely, we study how to use the information given by a precomputed optimal basis at a given $\lambda$ to warmstart the computation of optimal solutions for other values of $\lambda$.

Given a basis $B$, let $x_B$ be the basic variables and $x_N$ the nonbasic variables:
\begin{equation}
\begin{aligned}
    x_B(\lambda) &= (A_B + \lambda D_B)^{-1} b &
    x_N(\lambda) &= 0,
\end{aligned}
\end{equation}
where, for a matrix $M$, $M_B$ and $M_N$ denote its basic and nonbasic columns. Like $x_B$ and $x_N$, they are indexed by the variables $i \in B$ and $j \in N$, with unit vectors $e_i$ and $e_j$. The basis thus forms a potential solution to $\mathcal{P}(\lambda)$, at least if $A_B + \lambda D_B$ is invertible. A basis is said to be feasible for a given $\lambda$ if the matrix is indeed invertible and $x_B(\lambda) \geq 0$, and optimal if it provides an optimal solution. We denote by $o^*(\lambda)$ the optimal objective value of $\mathcal{P}(\lambda)$ and, for a basis $B$ such that $A_B + \lambda D_B$ is invertible, by
\begin{equation}
    o_B(\lambda) := c^{\top}_B(A_B + \lambda D_B)^{-1} b = c^{\top}_B x_B(\lambda)
\end{equation}
the objective value of its basic solution, whether this solution is feasible or not. If $B$ is optimal for $\mathcal{P}(\lambda)$, then $o^*(\lambda) = o_B(\lambda)$. The dual solution associated with $B$ is $y(\lambda)$, given by $y(\lambda)^{\top} = c_B^{\top}(A_B + \lambda D_B)^{-1}$.

The core of the paper revolves around finding methods to reduce the cost of solving the system $(A_B + \lambda D_B)x_B(\lambda)=b$ repeatedly for various $\lambda$, in order to compute $x_B(\lambda)$ and quickly check the optimality of the basis for different values of $\lambda$.

\subsection{Related work}
As mentioned above, most articles in the literature tackle the sensitivity analysis of the objective function $c$ or the right-hand side $b$ \citep{jansen,gal_multi,Gal-1994,obj_param1,obj_param2}. \citet{Lee.2019} solve large batches of linear programs that differ only in their right-hand side, by grouping the values that share an optimal basis; in this case, the basis matrix does not change. Only a few papers deal with the problem of assessing the behavior of the optimum value and the optimal solution for a varying constraint matrix. \citet{Sherman2,sherman} update the inverse of a matrix when the entries of one column or one row change, and when one entry changes. \citet{Bartlett.1951} extends the formula to any rank-one change $D=uv^{\top}$. For such a change, this formula turns the conditions for $B$ to remain optimal into linear inequalities in $\lambda$~\citep{Gal-1994}. \citet{Greenberg.1999} recalls this result and extends it to the optimal partition of an interior solution. \citet{woodbury} extends the formula of Sherman and Morrison to changes of any rank; Section~\ref{sec:lowrank} uses it to reduce all the computations to a small matrix when the perturbation has a low rank. More recently, \citet{nous} provide methods to compute upper and lower bounds on the objective function for our class of problems.
\citet{Freund.1985} gives the derivatives, and the Taylor series, of the optimal value and of the optimal basic solutions with respect to $\lambda$ at a nondegenerate optimal basis, and discusses the degenerate case; Section~\ref{sec:lr-radius} uses the first-order terms of these expansions to certify an interval of optimality, also for degenerate bases.

The algorithm most related to our contribution is presented in \citet{Zuidwijk}, where Zuidwijk provides an exact algorithm to calculate $o_B(\lambda)=c_B^{\top}(A_B + \lambda D_B)^{-1}b$. By using realization theory \citep[Chapter~3]{bart}, Zuidwijk derives the following formula:
\begin{equation}
o_B(\lambda) = \frac{1}{\lambda} \big(\prod_{j=1}^m \frac{1+\lambda\beta_j}{1+\lambda\alpha_j}-1\big) \label{eq:zuidwijk}
\end{equation}
where $\alpha_j$ and $\beta_j$ are the eigenvalues of $A_B^{-1}D_B$ and $A_B^{-1}(D_B+bc_B^{\top})$, respectively. To prove the correctness of this algorithm, the method checks three conditions. It computes the eigenvalues of one matrix for the existence condition, of $m$ matrices for the feasibility condition and of $2(n-m)$ matrices for the optimality condition. This preprocessing step costs $\mathcal{O}(nm^{\omega})$, assuming matrix-matrix multiplication is $\mathcal{O}(m^\omega)$.

Most operations on matrices (matrix-matrix multiplication, inversion, system solve, QR decomposition, eigenvalue problems, ...) can be done in $\mathcal{O}(n^{\omega+\eta})$ for any $\eta > 0$~\citep{Demmel_2007}; for simplicity, we write $\mathcal{O}(n^\omega)$. In practical applications, $\omega$ is generally $3$, but there exist galactic algorithms which can reduce $\omega$.%

\subsection{Our contribution}
We provide algorithms to compute the exact value of an optimal solution $x_B(\lambda)$ and the optimum value $o^*(\lambda)$ of problem $\mathcal{P}(\lambda)$ around an optimal basis $B$ for a given set of values $\lambda \in \Lambda$. The algorithms are based on three different reformulations of the term $(A_B + \lambda D_B)^{-1}$, which use an eigendecomposition, a Schur decomposition or a Hessenberg decomposition of $A_B^{-1}D_B$; the latter always exists, is computed without eigenvalues in real arithmetic, and its reduction is cheaper than the Schur decomposition.
 These reformulations limit the computational complexity of finding $x_B(\lambda)$ and $o_B(\lambda)$ for new $\lambda$. More precisely, they avoid the need to invert large matrices or solve large systems of unknown structure for every $\lambda\in\Lambda$. The three algorithms have different trade-offs, but all share a global runtime to compute $p=\lvert \Lambda \rvert$ different points $(x_B(\lambda),\ o_B(\lambda))$ of $\mathcal{O}(m^\omega + pm^2)$, and of $\mathcal{O}(m^\omega + pm^2 + pmn)$ with the optimality test that gives $o^*(\lambda)$, with $m$ (resp.\ $n$) the number of constraints (resp.\ variables) of $\mathcal{P}(\lambda)$. In addition, we present how each of these algorithms ensures that $B$ is feasible and optimal for each $\lambda \in \Lambda$.

We then make these methods applicable to large problems:
\begin{itemize}
    \item Section~\ref{sec:lowrank} shows that, when $A_B^{-1}D_B$ has a low rank $\rho$, as is often the case in practice, all the quantities above can be computed from a $\rho\times \rho$ matrix and from two $m\times\rho$ factors, built with $\rho$ solves with $A_B$ and $\rho$ with $A_B^{\top}$. These solves reuse the LU factorization of the simplex solver.
    \item With these low-rank factors, Section~\ref{sec:lr-radius} provides, in each direction, an interval around a given $\lambda$ on which the basis is guaranteed to remain optimal, together with a bound on the objective function on this interval. In the experiments, these intervals are close to the exact ones, and they remain useful when the basis is degenerate, which is common in practice.
    \item Section~\ref{sec:krylov} presents a matrix-free alternative, based on shifted Krylov subspaces, that needs one solve with the factorized basis matrix per iteration and terminates after at most $\rho+1$ iterations.
\end{itemize}
Finally, Section~\ref{sec:experiments} evaluates all the methods on \nInst{} instances built from real models, showing that both the low-rank and the Krylov-based techniques are accurate and that, as long as the basis remains optimal, they compute and certify $o^*(\lambda)$ \localSpeedupLrHess{} to \localSpeedupKrylov{} faster than a solver warm start, in median. Once the interval on which the basis is optimal is known, they compute $o^*(\lambda)$ \localNocertSpeedupLrHess{} to \localNocertSpeedupKrylov{} faster.

\section{The matrix $E=A_B^{-1}D_B$, its decompositions and optimality of the basis}\label{sec:decomposition_AB}\label{sec:hessenberg}

This section shows that one decomposition of the matrix $E := A_B^{-1}D_B$, computed once, reduces the cost of computing $x_B(\lambda)$ for a new value of $\lambda$ from $\mathcal{O}(m^\omega)$ to $\mathcal{O}(m^2)$ and how it allows us to efficiently check the three conditions for the basis $B$ to remain optimal.

\subsection{From $A_B + \lambda D_B$ to decomposing $E$}

We take $\lambda = 0$ as the nominal value and assume that $\mathcal{P}(0)$ has been solved: the solver provides an optimal basis $B$ and maintains an LU decomposition of $A_B$. When $A_B + \lambda D_B$ is invertible,
\begin{equation}
    x_B(\lambda) = (A_B + \lambda D_B)^{-1}b = (A_B(I_m + \lambda E))^{-1}b = (I_m + \lambda E)^{-1}x_B(0), \qquad x_B(0) = A_B^{-1}b. \label{eq:xb_fct_x0}
\end{equation}
Solving this dense system costs $\mathcal{O}(m^\omega)$~\citep{Demmel_2007} for each value of $\lambda$. A decomposition of $E$ avoids it.
\begin{proposition}\label{theorem:eigendecomposition}
    Let $E = QTQ^{-1}$ with $Q$ invertible. For any $\lambda$ such that $I_m + \lambda E$ is invertible,
    \begin{equation*}
        x_B(\lambda) = Q(I_m + \lambda T)^{-1}Q^{-1}x_B(0), \qquad o_B(\lambda) = c_B^{\top}Q(I_m + \lambda T)^{-1}Q^{-1}x_B(0).
    \end{equation*}
\end{proposition}
\begin{proof}{Proof}
    $I_m + \lambda E = Q(I_m + \lambda T)Q^{-1}$, and the formulas follow from~\eqref{eq:xb_fct_x0}.\Halmos
\end{proof}
The three following choices of $T$ are computed once, in $\mathcal{O}(m^\omega)$. They differ by their requirements and by the cost of the systems with $I_m + \lambda T$ (Table~\ref{tab:complexity}):
\begin{itemize}
    \item \textbf{Eigendecomposition}: $T$ is the diagonal matrix of the eigenvalues of $E$, so that $x_B(\lambda)$ costs $\mathcal{O}(m^2)$, and $o_B(\lambda)$ costs $\mathcal{O}(m)$ once $c_B^{\top}Q$ and $Q^{-1}x_B(0)$ are precomputed. The decomposition does not exist when $E$ is defective, which happens in practice.
    \item \textbf{Schur decomposition}: $E = QTQ^{H}$ with $Q$ unitary and $T$ upper triangular, so that back-substitution gives $x_B(\lambda)$ and $o_B(\lambda)$ in $\mathcal{O}(m^2)$. The decomposition always exists, but it requires complex arithmetic.
    \item \textbf{Hessenberg decomposition}: $E = PHP^{\top}$ with $P$ orthogonal and $H$ upper Hessenberg, computed directly by Householder reflections~\citep{GolubVanLoan.2013}, in real arithmetic and without eigenvalues. Gaussian elimination with partial pivoting solves the systems with $I_m + \lambda H$ in $\mathcal{O}(m^2)$, as only two rows are candidate pivots at each step and only one row is updated. Its pivots give $\det(I_m + \lambda H)$ up to the sign, so that a (numerically) zero pivot reveals that $I_m + \lambda E$ is singular.
\end{itemize}

\subsection{Optimality of the basis}\label{sec:basis_opt}
The methods above compute $x_B(\lambda)$. In the following, we check whether it is optimal.

Let $\nu_1, \dots, \nu_m$ be the eigenvalues of $E$. The eigenvalues of $I_m + \lambda E$ are the numbers $1 + \lambda\nu_k$, and $A_B + \lambda D_B = A_B(I_m + \lambda E)$ is thus invertible if and only if $\lambda \neq -1/\nu_k$ for every nonzero $\nu_k$. If the three following conditions hold, then $B$ is optimal:
\begin{description}
    \item[(\textbf{existence})] $\lambda \neq -1/\nu_k$ for every nonzero $\nu_k$, i.e.\ $A_B + \lambda D_B$ is invertible. These eigenvalues are computed once, or avoided with the pivots of the Hessenberg elimination;
    \item[(\textbf{feasibility})] $x_B(\lambda) \geq 0$, checked in $\mathcal{O}(m^2)$. Then $o_B(\lambda)$ is an upper bound on $o^*(\lambda)$, even if $B$ is not optimal;
    \item[(\textbf{optimality})] the reduced costs are nonnegative, which is sufficient, and also necessary when the basic solution is nondegenerate. With the decomposition, they cost $\mathcal{O}(m^2 + mn)$ per value of $\lambda$:
    \begin{equation}
        r(\lambda) = c^{\top}_N - c^{\top}_B(A_B + \lambda D_B)^{-1}(A_N + \lambda D_N) = c^{\top}_N - c_B^{\top}Q(I_m + \lambda T)^{-1}Q^{-1}A_B^{-1}(A_N + \lambda D_N).\label{eq:reduced-costs}
    \end{equation}
\end{description}

\subsection{An initial algorithm for warm-starting from the decomposition}\label{sec:summary}
The decomposition is computed once. Then, for each $\lambda \in \Lambda$ ($\lvert\Lambda\rvert = p$), the three conditions are checked, and the objective and the solution are computed. If $B$ is not optimal at a value of $\lambda$, $\mathcal{P}(\lambda)$ is re-solved with a warm start from $B$, and the procedure restarts with the new optimal basis for the remaining values of $\lambda$.

Table \ref{tab:complexity} summarizes the asymptotic complexity of all the decomposition methods presented in this paper, and compares it to two methods: the naive method of recomputing the problem in full for each $\lambda$ and the method presented in \citet{Zuidwijk}.

As presented, these decompositions are sadly of a quite similar size to the original constraint matrix $A_B$, which is $m \times m$. When the matrix $E$ has a low rank, the next section reduces the size of the decomposition and thus the cost of computing $x_B(\lambda)$ and $o_B(\lambda)$ for new values of $\lambda$.

\begin{table}
\TABLE
{Complexity of various methods to compute $o^*(\lambda)$ and $x_B(\lambda)$ for all $\lambda \in \Lambda$, $|\Lambda|=p$, in big-Oh ($\mathcal{O}(\cdot)$) asymptotic complexity.\label{tab:complexity}}
{\begin{tabular}{lrrrrrrr}
        \toprule
        Method & \rotatebox[origin=c]{90}{Preprocessing} & \rotatebox[origin=c]{90}{Existence} &  \rotatebox[origin=c]{90}{ Feasibility} & \rotatebox[origin=c]{90}{Optimality} & \rotatebox[origin=c]{90}{Objective} & \rotatebox[origin=c]{90}{Solution} & \rotatebox[origin=c]{90}{Total}\\
        \midrule
        Naïve solve 
            & / 
            & /
            & /
            & /
            & $\approx n^{\omega}$ 
            & $\approx n^{\omega}$
            & $\approx pn^{\omega}$\\
        Basis system solving 
            & $m^{\omega}$ 
            & $m$ 
            & $m^{\omega}$ 
            & $m^{2}+mn$  
            & $m^{\omega}$
            & $m^{\omega}$
            & $pm^{\omega}+pmn$\\
        \citet{Zuidwijk} 
            & $nm^{\omega}$ 
            & $m$ 
            & $m^2$
            & $mn$
            & $m$ 
            & N/A
            & $nm^{\omega} + pm^2 + pmn$\\
        Eigendecomposition 
            & $m^\omega$ 
            & $m$ 
            & $m^2$ 
            & $m^2+mn$ 
            & $m$
            & $m^2$
            & $m^\omega + pm^2 + pmn$\\
        Schur 
            & $m^\omega$ 
            & $m$ 
            & $m^2$ 
            & $m^2+mn$ 
            & $m^2$
            & $m^2$
            & $m^\omega + pm^2 + pmn$\\
        Hessenberg 
            & $m^\omega$ 
            & \,$m^2$ 
            & $m^2$ 
            & $m^2+mn$ 
            & $m^2$
            & $m^2$
            & $m^\omega + pm^2 + pmn$\\
        \bottomrule
    \end{tabular}}
{In the table, $n$ is the number of variables in the original problem and $m$ the number of constraints, with $m < n$ in general. $\mathcal{O}(n^\omega)$ is the complexity of multiplying two $n \times n$ matrices.}
\end{table}

\section{Exploiting the low rank of the perturbation}\label{sec:lowrank}

As formulated in Section~\ref{sec:decomposition_AB}, the decompositions are computed on the $m\times m$ matrix
$E=A_B^{-1}D_B$; forming and decomposing it as a dense matrix requires $\mathcal{O}(m^\omega)$ operations and
$\mathcal{O}(m^2)$ memory, which restricts these methods to problems with at most a few thousand constraints. In this section, we assume that $D_B$ has a low rank $\rho \ll m$. This is often the case in practice, as a parameter usually affects a limited part of the model: the nonzero entries of $D$ may form rows, columns, diagonals or combinations of them, and the rank of $D_B$ is small as soon as they span few rows or few columns. This section shows that all the
quantities of Section~\ref{sec:decomposition_AB} can be obtained from a small $\rho\times \rho$
matrix (with $\rho \ll m$) and two $m\times\rho$ factors, and that building them only requires $\rho$ solves with the basis matrix $A_B$ and $\rho$ with $A_B^{\top}$. The simplex method already maintains $A_B$ in factored form, so these solves are computed cheaply.

\subsection{Low-rank reformulation}\label{sec:lr-reform}

The matrix $D_B$ is made of the columns of $D$ of the basic variables: its rows are those of $D$, and its columns are indexed by the basic variables. Let $\mathcal{J} \subseteq B$ be the set of the basic variables whose column in $D$ is nonzero,
$n_D=|\mathcal{J}|$, and let $\mathcal{I}$ be the set of the rows of $D_B$ that contain a nonzero entry (a subset of those of $D$),
$m_D=|\mathcal{I}|$. The matrix $D_B$ is zero outside of the rows $\mathcal{I}$ and of the columns $\mathcal{J}$, so
that it admits factorizations
\begin{equation}
    D_B = U V^{\top}, \qquad U, V \in \mathbb{R}^{m\times \rho}, \qquad \ \rho \ge \operatorname{rank}(D_B). \label{eq:lr-factor}
\end{equation}
Let $S_{\mathcal{J}}$ (resp.\ $S_{\mathcal{I}}$) be the matrix made of the columns $\mathcal{J}$ (resp.\ $\mathcal{I}$) of the identity matrix. Two factorizations are available without any computation: $U = D_B S_{\mathcal{J}}$, $V = S_{\mathcal{J}}$ (then $\rho=n_D$),
and $U = S_{\mathcal{I}}$, $V = D_B^{\top} S_{\mathcal{I}}$ (then $\rho = m_D$), since the columns of $D_B$ outside
$\mathcal{J}$ and its rows outside $\mathcal{I}$ are zero, so that $D_B = D_B S_{\mathcal{J}}S_{\mathcal{J}}^{\top} =
S_{\mathcal{I}}S_{\mathcal{I}}^{\top} D_B$. A rank-revealing factorization of the
$m_D \times n_D$ nonzero block of $D_B$ (e.g.\ its singular value decomposition) gives the smallest possible value,
$\rho = \operatorname{rank}(D_B)$; we call such a factorization \emph{minimal}. The experiments use the smaller of the first two, with $\rho = \min(n_D, m_D)$.

Given a factorization~\eqref{eq:lr-factor}, let us define the following matrices and vectors that will be used in the subsequent analysis:
\begin{equation}
    \begin{aligned}
        \hat{U} &= A_B^{-1}U \in \mathbb{R}^{m\times \rho}, & \qquad \hat{V} &= (A_B^{\top})^{-1}V \in \mathbb{R}^{m\times \rho},\\
        C &= V^{\top}\hat{U} \in \mathbb{R}^{\rho\times \rho}, & \qquad w_c &= \hat{U}^{\top}c_B \in \mathbb{R}^{\rho},
    \end{aligned}
    \label{eq:lr-defs}
\end{equation}
for which $E = A_B^{-1}UV^{\top} = \hat{U}V^{\top}$ by construction.

\begin{theorem}\label{thm:lr}
    Let $D_B = U V^{\top}$ be a factorization~\eqref{eq:lr-factor}, with the notation~\eqref{eq:lr-defs}. For any
    $\lambda$, $A_B+\lambda D_B$ is invertible if and only if $I_\rho + \lambda C$ is invertible, and in that case
    \begin{equation}
        (A_B+\lambda D_B)^{-1} = A_B^{-1} - \lambda\, \hat{U}\, (I_\rho+\lambda C)^{-1}\, V^{\top} A_B^{-1}, \label{eq:lr-inv}
    \end{equation}
    \begin{equation}
        x_B(\lambda) = x_B(0) - \lambda\, \hat{U} (I_\rho + \lambda C)^{-1} V^{\top} x_B(0), \qquad o_B(\lambda) = o_B(0) - \lambda\, w_c^{\top} (I_\rho + \lambda C)^{-1} V^{\top} x_B(0). \label{eq:lr-x}
    \end{equation}
\end{theorem}
\begin{proof}{Proof}
Since $D_B = UV^{\top} = A_B\hat{U}V^{\top}$, $A_B+\lambda D_B = A_B (I_m + \lambda \hat{U} V^{\top})$ and, by Sylvester's
    determinant identity, $\det(I_m + \lambda \hat{U} V^{\top}) = \det(I_\rho + \lambda V^{\top} \hat{U})$: the two matrices are
    simultaneously invertible. Equation~\eqref{eq:lr-inv} is the Sherman--Morrison--Woodbury
    formula~\citep{woodbury,Hager.1989}.
        Multiplying~\eqref{eq:lr-inv} by $b$ on the right gives $x_B(\lambda)$ in~\eqref{eq:lr-x},
    since $A_B^{-1}b = x_B(0)$, and multiplying it by $c_B^{\top}$ on the left gives $o_B(\lambda)$, since
    $o_B(\lambda) = c_B^{\top}x_B(\lambda)$ and $c_B^{\top}\hat{U} = w_c^{\top}$.\Halmos
\end{proof}

Only the $\rho\times \rho$ matrix $I_\rho + \lambda C$ depends on $\lambda$, and so any of the decompositions of Section~\ref{sec:decomposition_AB} can be applied to $C$ instead of $E$, with $\rho$ instead of $m$: with $C = PHP^{\top}$ or $C = QTQ^{H}$, the systems with $I_\rho + \lambda C$ cost $\mathcal{O}(\rho^2)$ per value of $\lambda$ after a decomposition in $\mathcal{O}(\rho^\omega)$, instead of $\mathcal{O}(m^2)$ and $\mathcal{O}(m^\omega)$, and an eigendecomposition of $C$, when it exists, gives the objective in $\mathcal{O}(\rho)$.%

\subsection{Spectrum, existence and diagonalizability}

The spectra of $E$ and $C$ are closely related:

\begin{lemma}\label{lem:lr-spectrum}
    The nonzero eigenvalues of $E = A_B^{-1} D_B = \hat{U} V^{\top}$ are the nonzero eigenvalues of $C = V^{\top} \hat{U}$, with the same
    algebraic and geometric multiplicities. In particular, the existence condition of Section~\ref{sec:basis_opt}
    reads: $A_B + \lambda D_B$ is singular if and only if $\lambda = -1/\mu$ for a nonzero eigenvalue $\mu$ of $C$.
\end{lemma}

\begin{lemma}\label{lem:lr-diag}
    If the factorization~\eqref{eq:lr-factor} is minimal, then $E$ is diagonalizable if and only if $C$ is
    nonsingular and diagonalizable.
\end{lemma}
The proofs of these lemmas are in Appendix~\ref{app:proofs}. Lemma~\ref{lem:lr-diag} explains why the eigendecomposition of $E$ often fails to exist in practice (as we will see in Section~\ref{sec:experiments}): the
eigenvalue $0$ of $E$ has an algebraic multiplicity of at least $m - \rho$, and it is defective as soon as the
(minimal) matrix $C$ is singular, that is, as soon as $V^{\top}A_B^{-1}U$ loses rank. For instance, if $C=0$ but
$\hat{U}V^{\top}\neq 0$, then $E^2 = \hat{U} C V^{\top} = 0$ and $E$ is a nonzero nilpotent matrix. The low-rank formulation is less
demanding: Theorem~\ref{thm:lr} does not require $C$ to be nonsingular, and an eigendecomposition of $C$ only
requires $C$ to be diagonalizable ($C = 0$, for instance, is diagonalizable).

\subsection{Optimality conditions}

Inserting~\eqref{eq:lr-inv} in the first expression of~\eqref{eq:reduced-costs} gives, for every $\lambda$ such that $I_\rho + \lambda C$ is invertible,
\begin{equation}
    r(\lambda) = c_N^{\top} - \big(y(0)^{\top} - \lambda\, w_c^{\top} (I_\rho + \lambda C)^{-1} \hat{V}^{\top}\big) (A_N + \lambda D_N). \label{eq:lr-r}
\end{equation}
The vector in parentheses requires $\mathcal{O}(\rho^2 + m\rho)$ operations per value of $\lambda$, and the product with $A_N + \lambda D_N$ is sparse, i.e.\ $\mathcal{O}(\operatorname{nnz}(A_N) + \operatorname{nnz}(D_N))$ operations instead of the
$\mathcal{O}(m^2 + mn)$ of Section~\ref{sec:basis_opt}. The feasibility of $B$ for a value of $\lambda$ is thus checked
in $\mathcal{O}(\rho^2 + m\rho)$ and its optimality in $\mathcal{O}(\rho^2 + m\rho + \operatorname{nnz}(A_N) +
\operatorname{nnz}(D_N))$ operations.

\subsection{Complexity and implementation with a simplex solver}

To form $C = V^{\top}\hat{U}$, we need $\hat{U} = A_B^{-1}U$, which requires $\rho$ solves with $A_B$. The reduced costs~\eqref{eq:lr-r} also need $\hat{V} = (A_B^{\top})^{-1}V$, which requires $\rho$ solves with $A_B^{\top}$. Their
right-hand sides are sparse (the columns of $U$ are supported on the rows $\mathcal{I}$ and those of $V$ on the rows $\mathcal{J}$). LP solvers expose these solves with $A_B$ and $A_B^{\top}$, called FTRAN and BTRAN, so no factorization of $A_B$ has to be computed. Let $\tau$ be the cost of one such solve. This cost is proportional to the number of nonzeros of the LU factors of $A_B$, and much smaller than $m^2$ in practice. Table~\ref{tab:complexity-lr} compares the low-rank variants with the methods of Section~\ref{sec:decomposition_AB}.

\begin{table}
\TABLE
{Asymptotic complexity ($\mathcal{O}(\cdot)$) of the low-rank variants, where $\rho = \min(n_D,m_D)$ or $\rho = \operatorname{rank}(D_B)$, $\tau$ is
    the cost of a solve with the factorized basis matrix and $\mathrm{nz} = \operatorname{nnz}(A_N) +
    \operatorname{nnz}(D_N)$.\label{tab:complexity-lr}}
{\setlength{\tabcolsep}{4pt}\footnotesize
    \begin{tabular}{lllll}
        \toprule
        & Preprocessing & \multicolumn{3}{c}{Per value of $\lambda$}\\
        & & Objective & Solution & Optimality\\
        \midrule
        Dense $E$ (Sec.~\ref{sec:decomposition_AB}) & $m^\omega$ & $m$ or $m^2$ & $m^2$ & $m^2 + mn$\\
        Low-rank, eig.\ of $C$ & $2\rho\tau + \rho^\omega + m\rho$ & $\rho$ & $m\rho$ & $m\rho + \mathrm{nz}$\\
        Low-rank, Schur/Hess.\ of $C$ & $2\rho\tau + \rho^\omega + m\rho$ & $\rho^2$ & $\rho^2 + m\rho$ & $\rho^2 + m\rho + \mathrm{nz}$\\
        \bottomrule
    \end{tabular}}
{Forming $C = V^{\top}\hat{U}$ costs $\mathcal{O}(\operatorname{nnz}(V)\,\rho)$ more (nothing for
    $V = S_{\mathcal{J}}$), and the optional rank-revealing factorization $\mathcal{O}(n_D m_D\min(n_D,m_D))$. The
    eigendecomposition variant requires $C$ to be diagonalizable.}
\end{table}

\section{Local radius of optimality}\label{sec:lr-radius}

The methods above compute $o_B(\lambda)$ efficiently, and allow us to check whether the computed values are optimal, i.e. if locally $o_B(\lambda) = o^*(\lambda)$. A piecewise-linear plot of the computed values $o_B(\lambda_i)$, $\lambda_i \in \Lambda$, implicitly assumes that $o_B(\lambda)$ does not deviate much from this linear interpolation and that the basis remains optimal between two sampled points. In general, $o_B(\lambda)$ has no such property~\citep{nous}. In this section, we compute, with the low-rank factors of Section~\ref{sec:lowrank} and for a given optimal basis $B$ at $\lambda=0$, an interval $(-\Delta^-, \Delta^+)$ on which $B$ is guaranteed to remain optimal; we call $\Delta^-$ and $\Delta^+$ the \emph{local radii of optimality}. We also compute lower and upper bounds on $o^*(\lambda)$ on this interval, which bound in particular
\begin{equation*}
    \sup_{-\Delta^- < \lambda < \Delta^+} \lvert o^*(\lambda) - o^*(0) \rvert.
\end{equation*}

\subsection{First-order expansion}\label{sec:first-order}

The functions $x_B(\lambda)$, $r(\lambda)$ and $o_B(\lambda)$ are rational functions of $\lambda$. Their denominator, $\det(I_m + \lambda E)$, has degree up to $m$, so their extrema on an interval are hard to locate when $m$ is large. We therefore write each function as its first-order expansion at $\lambda = 0$, which is explicit, plus a remainder of order $\lambda^2$. Section~\ref{sec:remainders} bounds these remainders.

For the basic solution, let $W = (I_m + \lambda E)^{-1}$, so that $x_B(\lambda) = Wx_B(0)$ by~\eqref{eq:xb_fct_x0}. The identity $(I_m + \lambda E)W = W(I_m + \lambda E) = I_m$ gives $W = I_m - \lambda WE = I_m - \lambda EW$. Using the first form, then the second one, gives the exact equality
\begin{equation*}
    x_B(\lambda) = x_B(0) - \lambda WEx_B(0) = x_B(0) + \lambda\dot x_B - \lambda^2 EW\dot x_B, \qquad \dot x_B = -Ex_B(0).
\end{equation*}
Its first two terms form the first-order expansion of $x_B$ at $\lambda = 0$, with the derivative $\dot x_B$, and the last term is the remainder. With the low-rank notation~\eqref{eq:lr-defs}, $E = \hat{U}V^{\top}$ and $EW = \hat{U}(I_\rho + \lambda C)^{-1}V^{\top}$, so the remainder only involves the $\rho\times\rho$ matrix $C$. The same steps apply to $r(\lambda)$ and $o_B(\lambda)$.

\begin{lemma}\label{lem:first-order}
    Let $B$ be an optimal basis for $\mathcal{P}(0)$, with the notation~\eqref{eq:lr-defs}, and let
    \begin{equation*}
    \begin{aligned}
        \dot x_B &= -Ex_B(0), & G &= EA_B^{-1}A_N - A_B^{-1}D_N, & \dot r &= c_B^{\top}G, & \dot o_B &= c_B^{\top}\dot x_B,\\
        w_x &= V^{\top}\dot x_B, & \theta_j &= V^{\top}Ge_j = C\hat{V}^{\top}A_Ne_j - \hat{V}^{\top}D_Ne_j \quad (j \in N), & w_c &= \hat{U}^{\top}c_B. & &
    \end{aligned}
    \end{equation*}
    For every $\lambda$ such that $I_m + \lambda E$ is invertible, each basic variable $i \in B$ and each nonbasic variable $j \in N$,
    \def\colA{(x_B(\lambda))_i}\def\colB{(x_B(0))_i + \lambda\,(\dot x_B)_i}\def\colC{-\,\lambda^2\, (\hat{U}^{\top}e_i)^{\top}(I_\rho + \lambda C)^{-1}w_x,}%
    \begin{equation}
        \colR{\colA}{(x_B(\lambda))_i} = \colL{\colB}{(x_B(0))_i + \lambda\,(\dot x_B)_i} \quad \colL{\colC}{-\,\lambda^2\, (\hat{U}^{\top}e_i)^{\top}(I_\rho + \lambda C)^{-1}w_x,} \label{eq:rem-x}
    \end{equation}
    \begin{equation}
        \colR{\colA}{(r(\lambda))_j} = \colL{\colB}{(r(0))_j + \lambda\,(\dot r)_j} \quad \colL{\colC}{-\,\lambda^2\, \theta_j^{\top}(I_\rho + \lambda C^{\top})^{-1}w_c,} \label{eq:rem-r}
    \end{equation}
    \begin{equation}
        \colR{\colA}{o_B(\lambda)} = \colL{\colB}{o_B(0) + \lambda\, \dot o_B} \quad \colL{\colC}{-\,\lambda^2\, w_c^{\top}(I_\rho + \lambda C)^{-1}w_x.} \label{eq:rem-o}
    \end{equation}
\end{lemma}
The proof is in Appendix~\ref{app:proofs}.
Equations~\eqref{eq:rem-x}--\eqref{eq:rem-o} share the general form
\begin{equation}
    \underbrace{\delta + \lambda\dot\delta}_{\text{first-order expansion}} \underbrace{- \lambda^2\, \xi^{\top}(I_\rho + \lambda K)^{-1}\zeta}_{\text{remainder}}, \label{eq:rem-general}
\end{equation}
with the values of Table~\ref{tab:rem-general}.

\begin{table}
\TABLE
{Terms of the general form~\eqref{eq:rem-general} for each function of Lemma~\ref{lem:first-order}.\label{tab:rem-general}}
{\begin{tabular}{llccccc}
        \toprule
        Function & Eq. & $\delta$ & $\dot\delta$ & $K$ & $\xi$ & $\zeta$\\
        \midrule
        Basic variable $i$ & \eqref{eq:rem-x} & $(x_B(0))_i$ & $(\dot x_B)_i$ & $C$ & $\hat{U}^{\top}e_i$ & $w_x$\\
        Reduced cost $j \in N$ & \eqref{eq:rem-r} & $(r(0))_j$ & $(\dot r)_j$ & $C^{\top}$ & $\theta_j$ & $w_c$\\
        Objective & \eqref{eq:rem-o} & $o_B(0)$ & $\dot o_B$ & $C$ & $w_c$ & $w_x$\\
        \bottomrule
    \end{tabular}}
{}
\end{table}

\subsection{Bounding the remainders}\label{sec:remainders}

To certify a condition on a whole interval of values of $\lambda$, we bound its remainder in~\eqref{eq:rem-general} on this interval with quantities computed at $\lambda = 0$ only. This amounts to bounding the scalar function $\xi^{\top}(I_\rho + \lambda K)^{-1}\zeta$, with $\xi, \zeta \in \mathbb{R}^\rho$ and $K = C$ or $C^{\top}$. The Neumann series~\citep[Section~2.3]{GolubVanLoan.2013} and H\"older's inequality give $\lvert\xi^{\top}(I_\rho + \lambda K)^{-1}\zeta\rvert \le \lvert\xi\rvert_1\lvert\zeta\rvert_\infty/(1 - \lvert\lambda\rvert\lVert K\rVert_\infty)$ when $\lvert\lambda\rvert\lVert K\rVert_\infty < 1$, where $\lVert\cdot\rVert_\infty$ is the matrix norm induced by the vector norm $\lvert\cdot\rvert_\infty$. However, $\lVert K\rVert_\infty$ can be
orders of magnitude larger than the spectral radius of $K$. Since $(I_\rho + \lambda K)^{-1}\zeta$ only depends on the action of
$K$ on the vectors $\zeta, K\zeta, K^2\zeta, \dots$, the matrix $K$ can be replaced by its restriction to the Krylov subspace spanned by these vectors. This subspace often has a dimension much smaller than $\rho$. We call this replacement the \emph{reduction to the Krylov subspace}. Together with a diagonal scaling, it gives the following, sharper, bound. Section~\ref{sec:radius-choice} gives the scalings of our implementation.

\begin{lemma}\label{lem:reduced-bound}
    Let $K\in\mathbb{R}^{\rho\times \rho}$ and $\xi, \zeta\in\mathbb{R}^\rho$, let the columns of
    $Q\in\mathbb{R}^{\rho\times d}$ be an orthonormal basis of the Krylov subspace $\mathcal{K}_\rho(K,\zeta) =
    \operatorname{span}\{\zeta, K\zeta, \dots, K^{\rho-1}\zeta\}$, of dimension $d$, let $H =
    Q^{\top}K Q \in \mathbb{R}^{d\times d}$ be the matrix of the restriction of $K$ to this subspace in
    the basis $Q$, and let $\Omega$ be a positive diagonal matrix of order
    $d$ and $\gamma = \lVert\Omega^{-1}H\Omega\rVert_\infty$. For every $\lambda$ such that $I_\rho + \lambda K$
    is invertible and $\lvert\lambda\rvert\gamma < 1$,
    \begin{equation*}
        \lvert\xi^{\top}(I_\rho + \lambda K)^{-1}\zeta\rvert \le \frac{\lvert\Omega Q^{\top}\xi\rvert_1\,
        \lvert\Omega^{-1}Q^{\top}\zeta\rvert_\infty}{1 - \lvert\lambda\rvert\gamma},
    \end{equation*}
    where $\lvert\cdot\rvert_1$ is the vector $1$-norm.
\end{lemma}
The proof is in Appendix~\ref{app:proofs}.

\subsection{Radius of the basis}\label{sec:radius-basis}

Each condition of optimality of $B$ is a scalar inequality: a basic variable $(x_B(\lambda))_i$ or a reduced cost $(r(\lambda))_j$ must remain nonnegative. Without loss of generality, we consider $\lambda \ge 0$ (the case $\lambda \le 0$ follows Theorem~\ref{thm:radius}). The value of a condition (i.e. either \eqref{eq:rem-x} or \eqref{eq:rem-r}) is then given by~\eqref{eq:rem-general} and Table~\ref{tab:rem-general}, and its \emph{margin} $\delta$, i.e.\ its value at $\lambda = 0$, is nonnegative since $B$ is optimal for $\mathcal{P}(0)$. Lemma~\ref{lem:reduced-bound} bounds the remainder: with $a = \lvert\Omega Q^{\top}\xi\rvert_1\lvert\Omega^{-1}Q^{\top}\zeta\rvert_\infty$, it is at least $-a \lambda^2/(1 - \gamma \lambda)$ as long as $I_\rho + \lambda K$ is invertible and $\gamma \lambda < 1$. A condition (\eqref{eq:rem-x} or \eqref{eq:rem-r}) is therefore guaranteed to hold at $\lambda$ as soon as
\begin{equation*}
    f(\lambda) = \delta + \dot\delta\, \lambda - \frac{a\, \lambda^2}{1 - \gamma \lambda} \ge 0,
\end{equation*}
and we look for the largest $T$ such that $f \ge 0$ on the whole interval $[0, T)$: the condition then holds for every $\lambda \in [0, T)$. The following lemma gives this $T$ in closed form.

\begin{lemma}\label{lem:first-root}
    Let $\delta \ge 0$, $\dot\delta \in \mathbb{R}$, $a \ge 0$ and $\gamma \ge 0$, and let $f(t) = \delta + \dot\delta\, t - a t^2/(1 - \gamma t)$ for $t \in [0, 1/\gamma)$, with $1/0 = +\infty$. Let $t_q$ be the smallest root in $(0, 1/\gamma)$ of the quadratic polynomial $q(t) = \delta + (\dot\delta - \gamma\delta)t - (a + \gamma\dot\delta)t^2$, or $1/\gamma$ if $q$ has no root in $(0, 1/\gamma)$, and
    \begin{equation}
        \operatorname{rad}(\delta, \dot\delta, a; \gamma) =
        \begin{cases}
            0 & \text{if } \delta = 0 \text{ and } \big(\dot\delta < 0 \text{ or } (\dot\delta = 0 \text{ and } a > 0)\big),\\
            1/\gamma & \text{if } \delta = \dot\delta = a = 0,\\
            t_q & \text{otherwise.}
        \end{cases}
        \label{eq:rad}
    \end{equation}
    Then $\operatorname{rad}(\delta, \dot\delta, a; \gamma)$ is the largest $T \le 1/\gamma$ such that $f \ge 0$ on $[0, T)$.
\end{lemma}
The proof is in Appendix~\ref{app:proofs}.

Lemma~\ref{lem:first-root} gives each basic variable $i \in B$ and each reduced cost $j \in N$ its own radius, $R^x_i$ and $R^r_j$, equal to $\operatorname{rad}(\delta, \dot\delta, a; \gamma)$ with the quantities of the corresponding row of Table~\ref{tab:rem-general} and with a scaling potentially chosen independently for each of them. The basis matrix limits the radius as well: by Lemma~\ref{lem:lr-spectrum}, $A_B + \lambda D_B$ is singular if and only if $\lambda = -1/\mu$ for a real eigenvalue $\mu < 0$ of $C$, and we denote by $\eta^+$ the smallest of these values ($\eta^+ = +\infty$ if there is none).

\begin{theorem}\label{thm:radius}
    Let $B$ be an optimal basis for $\mathcal{P}(0)$ and, with the notation above,
    \begin{equation*}
        \Delta^+ = \min\Big(\eta^+,\ \min_{i \in B} R^x_i,\ \min_{j \in N} R^r_j\Big).
    \end{equation*}
    Then $B$ is optimal for $\mathcal{P}(\lambda)$ for every $\lambda \in [0, \Delta^+)$.
\end{theorem}
\begin{proof}{Proof}
Let $0 \le \lambda < \Delta^+$. Since $\lambda < \eta^+$, $A_B + \lambda D_B = A_B(I_m + \lambda E)$ is invertible, and so are $I_m + \lambda E$ and, by Theorem~\ref{thm:lr}, $I_\rho + \lambda C$. For each basic variable and each reduced cost, $\lambda$ is smaller than its radius, which is at most $1/\gamma$, so that Lemma~\ref{lem:reduced-bound} applies to its remainder and its value at $\lambda$ is at least $f(\lambda)$, which is nonnegative by Lemma~\ref{lem:first-root}. The basis $B$ is thus feasible and optimal for $\mathcal{P}(\lambda)$.\Halmos
\end{proof}

For $\lambda \le 0$, $\mathcal{P}(\lambda)$ is $\mathcal{P}(-\lambda)$ with $D$ replaced by $-D$. This changes the signs of $E$, $C$, $\dot\delta$, $\xi$ and $\zeta$, but not $a$ and $\gamma$. Theorem~\ref{thm:radius} thus gives the radius $\Delta^-$ with $-\dot\delta$ instead of $\dot\delta$ and the real eigenvalues $\mu > 0$ of $C$.

Every condition is handled separately, which makes the radius resistant to degeneracy. A degenerate condition, i.e.\ $\delta = 0$, gives $\operatorname{rad}(0, \dot\delta, a; \gamma) = \dot\delta/(a + \gamma\dot\delta) > 0$ if $\dot\delta > 0$, and a zero radius if $\dot\delta < 0$. The zero radius is the exact answer: the condition is then violated at first order, for any small $\lambda > 0$.

\begin{corollary}\label{cor:enclosure}
    With the notation of Theorem~\ref{thm:radius}, let $(a_o, \gamma_o)$ be the coefficients of
    Lemma~\ref{lem:reduced-bound} for the objective (Table~\ref{tab:rem-general}). For every $\lambda \in (-\Delta^-, \Delta^+)$ such that $\gamma_o \lvert\lambda\rvert < 1$,
    \begin{equation}
        \lvert o^*(\lambda) - o^*(0) - \lambda\, \dot o_B\rvert \le \frac{a_o\, \lambda^2}{1 - \gamma_o \lvert\lambda\rvert}. \label{eq:enclosure}
    \end{equation}
\end{corollary}
\begin{proof}{Proof}
    By Theorem~\ref{thm:radius}, $o^*(\lambda) = o_B(\lambda)$, and the bound follows from~\eqref{eq:rem-o} and Lemma~\ref{lem:reduced-bound}.\Halmos
\end{proof}

This bound gives the guarantee on the objective that motivated this section: for a tolerance $\epsilon > 0$, $\lvert o^*(\lambda) - o^*(0)\rvert \le \epsilon$ for all $\lambda \in (-\Delta^-, \Delta^+)$ such that $\lvert\lambda\rvert < \operatorname{rad}(\epsilon, -\lvert\dot o_B\rvert, a_o; \gamma_o)$, by Lemma~\ref{lem:first-root}.

\subsection{Choice of the scaling}\label{sec:radius-choice}

Lemma~\ref{lem:reduced-bound} holds for every positive diagonal scaling $\Omega$. The resulting radius depends on $\Omega$, and the best scaling can differ from one condition to another. As every scaling gives a valid radius, we
compute, for each condition, the radius with two classical scalings and keep the larger one:
\begin{itemize}
    \item the \emph{balancing} scaling~\citep{Osborne.1960,Parlett.1969}, computed by the LAPACK routine \texttt{xGEBAL},
    a general-purpose choice that reduces the norm of $\Omega^{-1}H\Omega$;
    \item the \emph{Perron} scaling~\citep{Safonov.1982} $\Omega = \operatorname{diag}(\psi)$, where the Perron vector $\psi$ of
    $\lvert H\rvert$ is its eigenvector with positive entries, associated with its largest eigenvalue: when $\lvert H\rvert$
    is irreducible, this vector exists~\citep[Chapter~8]{HornJohnson.2013} and gives the smallest possible value of
    $\gamma$.
\end{itemize}
Appendix~\ref{app:numerics} gives the details of these scalings, and Section~\ref{sec:exp-stability} measures their effect.

\subsection{Computation}\label{sec:radius-comp}

The radius only uses the low-rank factors of Section~\ref{sec:lowrank}: no solve with $A_B$ is needed beyond those that form
$\hat{U}$ and $\hat{V}$ (and, with bounded variables, the vector $h$ of Appendix~\ref{sec:lr-bounded}). Moreover, it does not need one Krylov
subspace per condition: the Arnoldi process (which builds an orthonormal basis of a Krylov subspace one vector at a time, and is recalled in Section~\ref{sec:krylov}) on $C$ started at $w_x$ serves all the basic variables and the objective, and the one on
$C^{\top}$ started at $w_c$ serves all the reduced costs. With $d$ and $d_c$ the dimensions of these two subspaces, usually
much smaller than $\rho$ (in the experiments of Section~\ref{sec:experiments}, the median of $d/\rho$ is \radDimPrimalMedian{} and that of $d_c/\rho$ is
\radDimDualMedian{}), and $\mathrm{nz} = \operatorname{nnz}(A_N) + \operatorname{nnz}(D_N)$, the radius requires
$\mathcal{O}(m\rho(d + d_c) + \mathrm{nz}\,d_c + d_c^2\lvert N\rvert)$ operations in total. Appendix~\ref{app:numerics} gives the details.

\section{Matrix-free evaluation with shifted Krylov subspaces}\label{sec:krylov}

This section computes $x_B(\lambda)$ without forming $E$ or the low-rank factors: the systems $(I_m + \lambda E)\, x = g$ for all the values of $\lambda$ are solved simultaneously in one Krylov subspace of small dimension, built with one solve with $A_B$ per dimension. This is useful when $\rho$ is large, since the low-rank method of Section~\ref{sec:lowrank} needs $\rho$ solves with $A_B$ and $\rho$ with $A_B^{\top}$ before the first value of $\lambda$ can be evaluated, and the dense methods need the full matrix $E$.

\subsection{Shifted Arnoldi process}

For every $\lambda$, $x_B(\lambda)$ is the solution of $(I_m + \lambda E)\,x = x_B(0)$. For a vector $g$, the Krylov
subspace $\mathcal{K}_j(E, g) = \operatorname{span}\{g, Eg, \dots, E^{j-1}g\}$ is the same for all the matrices $I_m +
\lambda E$: $\mathcal{K}_j(I_m + \lambda E, g) = \mathcal{K}_j(E, g)$ for $\lambda \neq 0$. A single basis of
$\mathcal{K}_j(E, x_B(0))$ therefore serves all the values of $\lambda$, which is the principle of the Krylov methods
for shifted systems~\citep{Datta.1991,Frommer.1998,Simoncini.2007}. This principle is classical. This section applies it to $E = A_B^{-1}D_B$ with the factorization of the solver, and shows that the low rank of $D_B$ bounds the number of steps (Theorem~\ref{thm:krylov-termination}). The Arnoldi process~\citep{Saad.2003} builds an orthonormal basis $Q_j = [v_1,\dots,v_j]$ of this subspace one vector at a time, starting from $v_1 = x_B(0)/\lVert x_B(0)\rVert_2$. Step $j$ computes $w = E v_j = A_B^{-1}(D_B v_j)$ with a sparse product with $D_B$ and one solve with $A_B$ (an FTRAN), so that $E$ is never formed. It then orthogonalizes $w$ against $v_1, \dots, v_j$ by Gram--Schmidt: for $i = 1, \dots, j$, $h_{i,j} = v_i^{\top}w$ and $w \leftarrow w - h_{i,j}v_i$. Finally, $h_{j+1,j} = \lVert w\rVert_2$ and $v_{j+1} = w/h_{j+1,j}$, and the process stops when $h_{j+1,j} = 0$. The coefficients $h_{i,j}$ form an
upper Hessenberg matrix $\underline{H}_j \in \mathbb{R}^{(j+1)\times j}$ such that
\begin{equation}
    E Q_j = Q_{j+1} \underline{H}_j = Q_j H_j + h_{j+1,j}\, v_{j+1} e_j^{\top},\label{eq:arnoldi}
\end{equation}
where $H_j = Q_j^{\top}EQ_j$ is made of the first $j$ rows of $\underline{H}_j$ and $h_{j+1,j}$ is its last entry.

We approximate $x_B(\lambda)$ by a vector of the subspace, $x_j(\lambda) = Q_j z_j(\lambda)$, whose coordinates
$z_j(\lambda) \in \mathbb{R}^j$ in the basis $Q_j$ are chosen so that the residual $x_B(0) - (I_m + \lambda E)\,
x_j(\lambda)$ is orthogonal to the subspace (the full orthogonalization method~\citep{Saad.2003}). Since $Q_j^{\top}
x_B(0) = \lVert x_B(0)\rVert_2\, e_1$ and, by~\eqref{eq:arnoldi}, $Q_j^{\top}(I_m + \lambda E)Q_j = I_j + \lambda H_j$,
this condition is the $j\times j$ system
\begin{equation}
    (I_j + \lambda H_j)\, z_j(\lambda) = \lVert x_B(0)\rVert_2\, e_1, \label{eq:fom}
\end{equation}
a Hessenberg system that is solved in $\mathcal{O}(j^2)$ for each value of $\lambda$ by Gaussian elimination with
partial pivoting~\citep{GolubVanLoan.2013}.

As in the full orthogonalization method~\citep[Section~6.4, Proposition~6.7]{Saad.2003}, \eqref{eq:arnoldi} gives, for every $\lambda$ such that $I_j + \lambda H_j$ is invertible, the residual
\begin{equation}
    x_B(0) - (I_m + \lambda E)\, x_j(\lambda) = -\lambda\, h_{j+1,j}\, \big(e_j^{\top} z_j(\lambda)\big)\, v_{j+1}. \label{eq:krylov-residual}
\end{equation}
Its norm $\lvert \lambda\, h_{j+1,j}\, (z_j(\lambda))_j \rvert$ is thus available for all the values of $\lambda$ at no extra cost. It does not tell when the approximation becomes exact. While for a general matrix the Arnoldi process may need up to $m$ steps, in our case the low rank of $D_B$ bounds the dimension of the Krylov subspace:

\begin{theorem}\label{thm:krylov-termination}
    $\dim \mathcal{K}_j(E, g) \le \operatorname{rank}(D_B) + 1$ for all $j$ and all $g$. Consequently, in exact
    arithmetic, the Arnoldi process on $E$ stops (with $h_{j+1,j} = 0$) after at most
    $\operatorname{rank}(D_B) + 1 \le \rho + 1$ steps, and then $x_j(\lambda) = x_B(\lambda)$ for every $\lambda$ such
    that $I_m + \lambda E$ is invertible (the eigenvalues of $H_j$ are then eigenvalues of $E$, so $I_j + \lambda H_j$ is invertible too).
\end{theorem}
\begin{proof}{Proof}
For $i \ge 1$, $E^i g = A_B^{-1}D_B(E^{i-1}g)$ lies in the range of $A_B^{-1}D_B$, of dimension
    $\operatorname{rank}(D_B)$; hence $\mathcal{K}_j(E,g) \subseteq \operatorname{span}\{g\} +
    \operatorname{range}(A_B^{-1}D_B)$. The process stops when $\mathcal{K}_{j}$ becomes invariant under $E$, i.e.\
    when $h_{j+1,j} = 0$; $H_j$ is then the restriction of $E$ to this invariant subspace, and the residual~\eqref{eq:krylov-residual} vanishes.\Halmos
\end{proof}

The dual solution satisfies $(I_m + \lambda \widetilde{E})\, y(\lambda) = y(0)$ with $\widetilde{E} = (A_B^{\top})^{-1}
D_B^{\top}$, so that the same process, applied to $\widetilde{E}$ with one solve with $A_B^{\top}$ (a BTRAN) per step, gives $y(\lambda)$ and the
reduced costs.
The objective only
needs the vector $c_B^{\top} Q_j$, and requires $\mathcal{O}(j^2)$ operations per value of $\lambda$.

The eigenvalues of $H_j$ (the Ritz values) only include the eigenvalues of $E$ reached from $x_B(0)$, so the Krylov process does not provide a complete existence test. For this, it is easier to use the eigenvalues of $C$ as computed in Section~\ref{sec:lowrank}.

\subsection{Adaptive evaluation}

The residual~\eqref{eq:krylov-residual} gives an adaptive algorithm: the subspace is extended, one solve with $A_B$ at a time, until the
relative residual of every requested value of $\lambda$ is below a tolerance ($10^{-12}$ in our experiments), and it
can be extended later if new values of $\lambda$ are requested; a value of $\lambda$ whose primal or dual residual has
not converged is never declared optimal. The cost is $j$ solves with $A_B$, $j$ sparse
products with $D_B$ and $\mathcal{O}(mj^2)$ operations for the orthogonalization, plus $\mathcal{O}(j^2 + mj)$ per
value of $\lambda$; Theorem~\ref{thm:krylov-termination} guarantees $j \le \rho+1$ in exact arithmetic. In practice, $j$
is much smaller: a median of \krylovItersMedian{} steps in the experiments of Section~\ref{sec:experiments}, against a median of \rMedian{} for $\rho$.

\section{Numerical experiments}\label{sec:experiments}

This section evaluates the methods of Section~\ref{sec:decomposition_AB}, together with the
low-rank and Krylov variants of Sections~\ref{sec:lowrank} and~\ref{sec:krylov}, on \nInst{} parametric linear programs built from real
models. We address five questions: (i) do the decompositions exist numerically for the optimal bases of real problems, and how accurate are their results; (ii) what do the
preprocessing and the evaluation of one value of $\lambda$ cost, compared with re-solving $\mathcal{P}(\lambda)$;
(iii) on which range of $\lambda$ does an optimal basis remain optimal, and how much of it does the local radius of
Section~\ref{sec:lr-radius} certify; (iv) is reusing the optimal bases faster than re-solving when computing $o^*(\lambda)$ on a grid of values
of $\lambda$; and (v) how much faster are the methods on the stability interval of an optimal basis, with and without a proof of optimality at each value. Implementation details, including the baselines (cold solves, solver warm starts, sparse LU) and the hardware, are given in Appendix~\ref{app:numerics}. The code and the scripts that reproduce all the experiments are publicly available together with the dataset~\citep{dataset}.

\subsection{Instances}\label{sec:exp-instances}

The instances come from the dataset of \citet{dataset}, with five families: the LPs of Netlib and of MIPLIB~2017 with randomly perturbed coefficients, capacitated facility location, unit commitment, and real models where $D$ describes a physically meaningful change. They have up to \mMax{} constraints, and the mixed-integer models are taken as their LP relaxations (Appendix~\ref{app:instances}). The nominal value $\lambda_0$ is the middle of the range of $\lambda$ (the instances are shifted so that $\lambda_0$ becomes $0$, as in the rest of the paper), and $\Lambda$ is a grid of $p$ equally spaced values of this range, with $p = \pGrid$ unless stated otherwise.

\subsection{Accuracy and robustness of the decompositions}\label{sec:exp-accuracy}

We compare, for the optimal basis $B$ of $\mathcal{P}(\lambda_0)$ and the \pGrid{} values of $\Lambda$, the results of each method with those of a sparse LU factorization of $A_B + \lambda D_B$, on the values of $\lambda$ for which $B$ is optimal (the dense methods are run on the \nInstSmall{} instances with $m \le \denseMax$). All the methods behave correctly: the error is at most \accMaxObj{} on the objective and \accMaxX{} on $x_B(\lambda)$, and the optimality tests always agree with the reference. There are two exceptions. The eigendecomposition of $E$ does not exist (numerically) for \eigFail{} of the \nInstSmall{} instances (\eigFailPct{}), confirming that defective matrices $E$ are the rule rather than the exception on real problems (Lemmas~\ref{lem:lr-spectrum} and~\ref{lem:lr-diag} explain why), and that of $C$ does not exist for \lrEigFail{} of the \nInst{} instances. Zuidwijk's formula, which only gives the objective, reaches an error of \zuidMaxErr{}.

\subsection{Computational cost}\label{sec:exp-cost}

Table~\ref{tab:exp-timing} reports the median preprocessing time of each method, the median time to compute
$o_B(\lambda)$ for one value of $\lambda$ (amortized over 1000 values), and the median time to compute $o_B(\lambda)$
\emph{and} certify that $B$ is optimal for $\lambda$ (existence, feasibility and optimality tests), which is the
quantity to compare with one re-solve of $\mathcal{P}(\lambda)$. Figure~\ref{fig:exp-cost}
shows how these times grow with $m$.

\begin{table}
\TABLE
{Median times (one thread).\label{tab:exp-timing}}
{\small
    \input{tables/tab_timing.tex}}
{Preprocessing, objective and certified $o^*(\lambda)$ per value of $\lambda$; the dense methods only run on the left part. The speed-up is the median, over
    the instances, of the ratio between the time of a solver warm start and the time to compute and certify
    $o_B(\lambda)$ (the eigendecomposition methods are only timed on the instances where they exist). The times of the Krylov method per value of $\lambda$ include its Arnoldi steps, amortized over the values evaluated; its preprocessing does not.}
\end{table}

\begin{figure}
\FIGURE
{\subcaptionbox{Preprocessing time of the Schur method (dense), of the low-rank method with a Hessenberg decomposition of $C$ and of the Krylov method, compared with the time to solve $\mathcal{P}(\lambda)$ from scratch.\label{fig:exp-prep}}
{\includegraphics[width=0.48\linewidth]{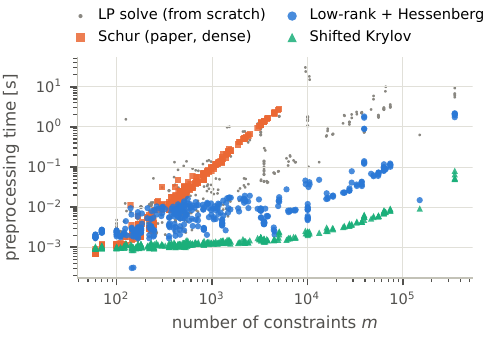}}
\hfill\subcaptionbox{Time to compute $o_B(\lambda)$ for one value of $\lambda$, without the optimality test, compared with the time of a solver warm start of $\mathcal{P}(\lambda)$.\label{fig:exp-eval}}
{\includegraphics[width=0.48\linewidth]{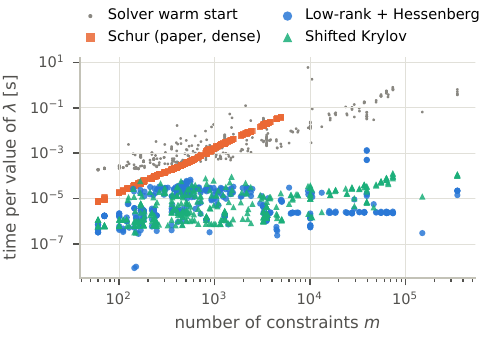}}}
{Times of the methods against the number of constraints $m$.\label{fig:exp-cost}}
{}
\end{figure}

The dense methods follow their $\mathcal{O}(m^2)$ per-$\lambda$ complexity: the Schur method is faster than a solver warm start on all the instances with $m \le 200$, but slower on \slowerSchurLarge{} of those with $m > 1000$.

The preprocessing of the low-rank method amounts to its $2\rho$ solves and the decomposition of $C$; the Krylov method only extracts $D_B$ in its preprocessing, and its evaluation times include its solves. Both methods evaluate $o_B(\lambda)$ in a few
microseconds in median (at most \objLrHessMax{} for the low-rank method and \objKrylovMax{} for the Krylov method
over all the instances), two to three orders of magnitude faster than a solver warm start. Certifying the optimality of $B$ is more
expensive, because the reduced costs require a product with $A_N$ (Equation~\eqref{eq:lr-r}); it remains
\certSpeedupLrHess{} (low-rank) and \certSpeedupKrylov{} (Krylov) faster than a solver warm start, in median.

These costs are low because the dimensions involved are small. The low-rank matrix $C$ has a median dimension of \rMedian{} (at most \rMax{}), and the Krylov method reaches a
relative residual of $10^{-12}$ for the \pGrid{} values of $\lambda$ of the whole range after a median of
\krylovItersMedian{} and at most \krylovItersMax{} steps (largest over the primal, dual and, when $h \neq 0$, $h$ subspaces of Appendix~\ref{sec:lr-bounded}), usually far below the
bound $\rho+1$ of Theorem~\ref{thm:krylov-termination} (the median ratio to $\rho$ is $0.24$, and the bound is only reached on 8 instances, all with $\rho \le 10$).

\subsection{Stability of the optimal basis and local radius}\label{sec:exp-stability}

The methods compute $o^*(\lambda)$ exactly only where $B$ remains optimal. For the optimal basis $B_0$ of
$\mathcal{P}(\lambda_0)$, we compute the exact stability interval, i.e.\ the largest interval containing $\lambda_0$
on which $B_0$ is optimal (with the tolerances of Appendix~\ref{app:numerics}, and up to the resolution of the scan), by scanning the range and bisecting on the first value where $B_0$ is not optimal (the
fast evaluation of Section~\ref{sec:lowrank} makes this cheap: \radExactTimeMedian{} in median), and the certified
radii $\Delta^-$ and $\Delta^+$ of Theorem~\ref{thm:radius}, denoted $\Delta^\sigma$ with $\sigma \in \{-, +\}$. We call \emph{sides} of the interval the distances $\Delta^-_{\mathrm{ex}}$ and $\Delta^+_{\mathrm{ex}}$ from $\lambda_0$ to its two ends; sides and radii are limited to the range of $\lambda$.

The stability intervals are small: their median width is \intWidthMedian{} for ranges of width $2$ ($0.3$ for the
facility location instances), and $B_0$ is optimal on a median of \coverMedian{} of the \pGrid{} values of $\Lambda$.
Exactly one of their sides is shorter than $10^{-6}$ for \intOneSidedPct{} of the instances, and both of them for
\intZeroSidedPct{}: the basis is degenerate, and a basic variable or a reduced cost leaves its bound as soon as
$\lambda$ moves in one (or either) direction, so that the median of the smaller side, $\min(\Delta^-_{\mathrm{ex}},
\Delta^+_{\mathrm{ex}})$, is only \intRadiusMedian{}. Degeneracy is the rule: \degPrimalPct{} of the optimal bases have
a basic variable at one of its bounds.

\begin{table}
\TABLE
{Certified radius $\Delta^\sigma$ of the optimal basis $B_0$ of $\mathcal{P}(\lambda_0)$ against the
    corresponding side $\Delta^\sigma_{\mathrm{ex}}$ of its exact stability interval, on the sides longer than
    $10^{-6}$.\label{tab:exp-radius}}
{\small
    \input{tables/tab_radius.tex}}
{Number of instances and of sides; median ratio between the radius computed as in
    Section~\ref{sec:lr-radius} and the exact side, and share of the sides where it reaches $1/2$ and $1/10$ of the
    exact side; same share ($1/2$) with the identity scaling only, without the reduction to
    the Krylov subspaces, and with the norms of $E$ instead of the low-rank coefficients (on the
    \radDenseNSides{} sides where they could be computed).}
\end{table}

Table~\ref{tab:exp-radius} compares the certified radius $\Delta^\sigma$ with the
corresponding side $\Delta^\sigma_{\mathrm{ex}}$ of the exact interval. The radius detects all the \radNTrivial{} sides shorter than $10^{-6}$, on which $B_0$ does not satisfy the optimality conditions in that direction. On the
\radNSides{} longer ones, it is a median of \radRatioMedian{} times the exact side, at least half of it on \radHalfPct{} of
them and at least a tenth on \radTenthPct{} (\radNEndSides{} of them end at the limit of the range). It costs
\radTimeMedian{} in median on top of the low-rank preprocessing (\radTimeSpeedup{} less than the exact interval). We check the radii numerically, by evaluating the optimality conditions at 40 points of each certified interval,
including its ends and its midpoint: they hold up to rounding errors.

The reduction to the Krylov subspaces and the scalings both matter (Table~\ref{tab:exp-radius}): without the reduction, the radius reaches half of the exact side on \radNoRedHalfPct{} of
the sides instead of \radHalfPct{}. With a simpler bound that uses the norms of the $m\times m$ matrix $E$, i.e.\ $\gamma = \lVert E\rVert_\infty$, without reduction or scaling, it reaches half of the
exact side on only \radDenseHalfPct{} of the \radDenseNSides{} sides where they could be computed.

\subsection{Computing \texorpdfstring{$o^*(\lambda)$}{o*(lambda)} on a grid}\label{sec:exp-curve}

We finally compute $o^*(\lambda)$ for all $\lambda \in \Lambda$ with the procedure of Section~\ref{sec:summary}:
starting from $B_0$, the values of $\Lambda$ are swept from $\lambda_0$ towards each end of the range; while the
current basis is optimal, $o^*(\lambda)$ is given by the method, and on the first value where it is not,
$\mathcal{P}(\lambda)$ is re-solved with a warm start from the current basis and the method restarts from the new
basis. The values are evaluated in blocks of increasing size, so that a basis that is quickly abandoned costs little.
All the computed values agree with the solver: they differ by at most \curveMaxErrSolver{} from its warm-started or its cold answer. The largest difference with the cold answer alone, \curveMaxErr{}, is that of the solver itself, whose warm-started and cold answers differ by up to \warmMaxErr{} (on instance \texttt{\warmMaxErrInstance}). $\mathcal{P}(\lambda)$ is infeasible on 10 values of each of 4 facility instances, and the solver warm start does not return an optimal solution on 4 values of 2 other instances, where a cold solve does: these values have no result. On the finer grids below, the computed values differ by at most \gridMaxErr{} from the solver warm starts.
Figure~\ref{fig:exp-curve} illustrates the procedure on one instance, and Table~\ref{tab:exp-curve} summarizes the
results for $p = \pGrid$; Table~\ref{tab:exp-grid} repeats the experiment with up to $p = \gridPLarge$ values, on the \gridNInst{} instances with $m \le 10^5$.

\begin{figure}
\FIGURE
{\includegraphics[width=0.6\linewidth]{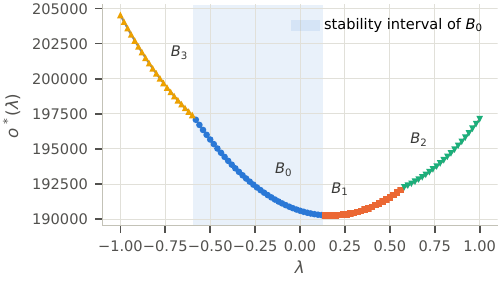}}
{$o^*(\lambda)$ for the instance \texttt{\curveInstance} on the grid of \pGrid{} values: each color and marker shape corresponds to one optimal basis; the bases are numbered in the order of their discovery, and the low-rank method computes their values; the shaded area is the exact stability interval of $B_0$.\label{fig:exp-curve}}
{}
\end{figure}

\begin{table}
\TABLE
{Computation of $o^*(\lambda)$ on the grid of $p = \pGrid$ values.\label{tab:exp-curve}}
{\small
    \input{tables/tab_curve.tex}}
{Median number of optimal
    bases per value of $\lambda$, median number of values covered by $B_0$ and median radius of its stability
    interval, and median total times relative to $p$ cold re-solves, for the solver warm starts and for the
    basis-reuse procedure with the low-rank (Hessenberg), Krylov and LU methods.}
\end{table}

\begin{table}
\TABLE
{Computation of $o^*(\lambda)$ on grids of $p$ values, on the \gridNInst{} instances with $m \le 10^5$.\label{tab:exp-grid}}
{\small
    \input{tables/tab_grid.tex}}
{Median number of changes of basis and of distinct optimal bases along the sweep of the Krylov method, median total times, and median ratio between the time of the basis-reuse procedure and that of the solver warm starts (below 1: basis reuse is faster), with the fraction of the instances where basis reuse with the Krylov method is faster.}
\end{table}

The time of the solver warm starts grows linearly with $p$ (\gridGrowthWarm{} from $p = \gridPSmall$ to $p = \gridPLarge$, in median), since each value needs a re-solve, whereas that of basis reuse grows much more slowly (\gridGrowthKrylov{} with the Krylov method), since one basis serves many values. The Krylov method takes a median of \gridRatioKrylovSmall{}, \gridRatioKrylovMid{} and \gridRatioKrylovLarge{} times the time of the solver warm starts for $p = \gridPSmall$, \gridPMid{} and \gridPLarge{}, and it is faster on \gridWinPctKrylovSmall{}, \gridWinPctKrylovMid{} and \gridWinPctKrylovLarge{} of the instances, by up to \gridSpeedupMaxKrylovLarge{} (Table~\ref{tab:exp-grid}).

The number of distinct bases is the key quantity. The optimal bases met along the range of $\lambda$ are essentially a property of the instance, but a grid only visits those whose stability interval contains one of its values: a coarse grid misses the short intervals, so that the number of bases grows with $p$ until the grid resolves them, whereas the cost of the re-solves keeps growing linearly with $p$. The median number of distinct bases is \gridDistinctSmall{}, \gridDistinctMid{} and \gridDistinctLarge{} for $p = \gridPSmall$, \gridPMid{} and \gridPLarge{}. It stops growing on the Facility instances and, in median, on the MIPLIB instances: a coarse grid already visits most of their bases. It still grows slowly on the Netlib and real-model instances. The unit commitment instances have very short stability intervals: with $p = \gridPLarge$, they still need a new basis at more than one value out of two, and basis reuse remains slower on almost all of them for all $p$. With $p = \pGrid$, a median
of \basesPerPoint{} bases per value of $\lambda$ is needed: the curves of many instances have a breakpoint every few values (a median of \netlibBasesPerPoint{} bases per value
for the Netlib instances, whose perturbation touches 100 constraints),
and the solver warm start, which needs only a few pivots per value, is then hard to beat. With $p = \gridPLarge$, the outcome is governed by the number of distinct bases per value of $\lambda$: basis reuse with the Krylov method is faster on \gridWinFewBases{} of the instances that need at most one new basis every 20 values, and on \gridWinManyBases{} of those that need more than one every 5 values. With $p = \pGrid$, the LU baseline is slower than the Krylov method on all the instances, which confirms that the gain comes from avoiding a factorization per value
of $\lambda$.

\subsection{Local analysis}\label{sec:exp-local}

The methods of this paper do not replace the re-solve at each change of basis.
Their natural use is the analysis of $o^*$ in the neighborhood of a solved problem, on the stability interval of its
optimal basis. Table~\ref{tab:exp-local} reports the speed-up, with respect to a solver warm start, of computing
and certifying $o^*(\lambda)$ for $\pLocal$ values of $\lambda$ inside the exact stability interval of $B_0$
(the solver warm starts start from $B_0$ and need no pivot there, except on 2 instances).

\begin{table}
\TABLE
{Local analysis, with and without the optimality test.\label{tab:exp-local}}
{\small
    \input{tables/tab_local.tex}}
{Median time of a solver warm start per value of $\lambda$ in the stability interval of $B_0$, and median speed-up of each method for computing $o^*(\lambda)$ at \pLocal{} values of the interval, preprocessing included, with the optimality test at each value and without it (then without the eigenvalues of the existence test). Without the test, the LR+eig.\ column only covers the \localNocertNLrEig{} instances where the decomposition exists.}
\end{table}

Within the stability interval, $o^*(\lambda)$ is obtained \localSpeedupKrylov{} (Krylov) and
\localSpeedupLrHess{} (low-rank) faster than with a solver warm start, in median, with a certificate of
optimality. Almost all of this time is the optimality test (Table~\ref{tab:exp-timing}), which is not needed at each value: the radius of Section~\ref{sec:lr-radius} or the exact stability interval is computed once and certifies all the values inside it. The right part of Table~\ref{tab:exp-local} reports the speed-up without this test: when the interval is known, it reaches \localNocertSpeedupLrHess{} (low-rank) and \localNocertSpeedupKrylov{} (Krylov) in median, preprocessing included. With the one-time cost of the interval and of the low-rank factors, amortized over the \pLocal{} values, it is \localNocertRadiusSpeedupLrHess{} (low-rank) and \localNocertRadiusSpeedupKrylov{} (Krylov) with the radius (\radTimeMedian{} in median on top of the factors), and \localNocertExactSpeedupLrHess{} and \localNocertExactSpeedupKrylov{} with the exact interval (\radExactTimeMedian{}). These speed-ups hold inside the interval only, and they depend on the number of values over which the preprocessing is amortized. The dense methods of Section~\ref{sec:decomposition_AB} are only competitive
on the smallest instances (the instances on which the Schur method is faster than a solver warm start have a median of \localSchurFasterMedianM{} constraints, against \localSchurSlowerMedianM{} for those on which it is slower).

\subsection{Discussion}

The low-rank and Krylov methods handle instances with hundreds of thousands of constraints, with the factorization already maintained by the solver. The Krylov method is the method of choice for evaluating $o^*(\lambda)$ at many values of $\lambda$: it only builds the necessary dimensions (a median of \krylovItersMedian{} steps, against a median dimension of \rMedian{} for $C$) and never forms the low-rank factors. The low-rank method is preferable when a certificate of optimality on an interval or a complete existence test is required, since its factors give the local radius of Section~\ref{sec:lr-radius} and the eigenvalues of $C$ give all the values of $\lambda$ at which $A_B + \lambda D_B$ is singular.

\section{Conclusion}

For linear programs $\mathcal{P}(\lambda)$ whose constraint matrix varies linearly with a parameter $\lambda$,
this paper reuses an optimal basis $B$ of $\mathcal{P}(0)$ to compute $x_B(\lambda)$ and $o_B(\lambda)$ for many
values of $\lambda$, together with a proof that $B$ remains optimal. When the perturbation of the basis matrix has a low rank $\rho$, all the computations reduce to a $\rho\times\rho$ matrix built with a few solves with
the factorization maintained by the simplex solver; shifted Krylov subspaces give a matrix-free alternative, which is
cheaper when only $o_B(\lambda)$ is needed. Finally, the low-rank factors give, in each direction, a radius up to which
$B$ is certified to remain optimal, with a bound on the objective; this radius handles
degenerate bases.

On \nInst{} instances built from real models, with up to \mMax{} constraints, the low-rank and Krylov methods compute $o^*(\lambda)$ on the stability interval of a basis \localSpeedupLrHess{} to \localSpeedupKrylov{} faster than a solver warm start with a certificate of optimality at each value, and \localNocertSpeedupLrHess{} to \localNocertSpeedupKrylov{} faster once the interval is known, in median and preprocessing included. Over the whole range of $\lambda$, reusing the bases becomes faster than solver warm starts as the grid gets finer: with \gridPLarge{} values, the Krylov method takes a median of \gridRatioKrylovLarge{} times their time, unless the basis changes at almost every value. The certified radius reaches at least half
of the exact stability interval on \radHalfPct{} of the sides longer than $10^{-6}$. As optimal bases of real problems
typically remain optimal on short ranges of $\lambda$, the methods are best used for local analyses and dense samplings
of $\lambda$, combined with solver warm starts at the breakpoints of $o^*(\lambda)$.

\section*{Agentic AI usage disclosure}
This article was written with the help of multiple AI tools, including agentic AI systems, after an initial draft by the authors. The authors supervised all interactions with these tools and verified the content produced, agree with its content and take responsibility for its correctness.

\bibliographystyle{informs2014}
\bibliography{ref,ref_extra}

\clearpage
\begin{APPENDICES}

\section{Bounded variables}\label{app:bounded}

The paper considers the standard form of Section~\ref{sec:intro}, where the nonbasic variables are zero. This
appendix extends the low-rank formulas of Section~\ref{sec:lowrank} and the local radius of
Section~\ref{sec:lr-radius} to bounded variables. Solvers handle linear programs in this form.

\subsection{Low-rank formulas}\label{sec:lr-bounded}

With bounds $\ell \le x \le u$, the nonbasic variables of an optimal basis sit at one of their bounds rather than at
zero. With $x_N$ the values of the nonbasic variables (slack variables included, with columns
of the identity), we redefine $x_B(0) = A_B^{-1}(b - A_N x_N)$, and the basic solution is
\begin{equation*}
\begin{aligned}
    x_B(\lambda) &= (A_B + \lambda D_B)^{-1}\left(b - (A_N + \lambda D_N)\, x_N\right)\\
                 &= (I_m + \lambda E)^{-1} g(\lambda), \qquad g(\lambda) = x_B(0) + \lambda h, \qquad
                    h = -A_B^{-1} D_N x_N.
\end{aligned}
\end{equation*}
The results of Section~\ref{sec:lowrank} hold with $x_B(0)$ replaced by $g(\lambda)$: $x_B(\lambda) = g(\lambda) - \lambda \hat{U}(I_\rho + \lambda C)^{-1} V^{\top} g(\lambda)$, and the objective is $c^{\top}x(\lambda) = c_N^{\top}x_N +
c_B^{\top}g(\lambda) - \lambda w_c^{\top}(I_\rho + \lambda C)^{-1} V^{\top} g(\lambda)$. The vector $h$ costs one more solve with $A_B$ (it is zero in the
setting of the paper, where $x_N = 0$). Feasibility becomes
$\ell_B \le x_B(\lambda) \le u_B$, and optimality requires $r_j(\lambda) \ge 0$ for the nonbasic variables at their
lower bound, $r_j(\lambda) \le 0$ for those at their upper bound and $r_j(\lambda) = 0$ for free nonbasic variables. The fixed variables ($\ell_j = u_j$, e.g.\ the slack variables of equality constraints) impose no condition. In the Krylov method of Section~\ref{sec:krylov}, the term $\lambda h$ of $g(\lambda)$ is handled by a second Krylov subspace started at $h$. All the
experiments of Section~\ref{sec:experiments} use this general form.

\subsection{Local radius of optimality}\label{sec:radius-bounded}

With bounded variables, the nonbasic variables are fixed at one of their bounds, $x_N$, and $x_B(\lambda) = (I_m + \lambda E)^{-1}(x_B(0) + \lambda h)$ with $x_B(0) = A_B^{-1}(b - A_Nx_N)$ and $h = -A_B^{-1}D_Nx_N$ (Appendix~\ref{sec:lr-bounded}). Lemma~\ref{lem:first-order}, Theorem~\ref{thm:radius} and Corollary~\ref{cor:enclosure} then hold with $\dot x_B = h - Ex_B(0)$, since $x_B(\lambda) - x_B(0) = \lambda(I_m + \lambda E)^{-1}(h - Ex_B(0))$, and each finite bound of a basic variable and each sign condition on a reduced cost is a condition with its own margin (the distance to the bound, or $\lvert r_j(0)\rvert$) and its own derivative ($\pm(\dot x_B)_i$ or $\pm(\dot r)_j$); the fixed nonbasic variables impose no condition.

\section{Numerical and implementation details}\label{app:numerics}

All the methods are implemented in Python, with NumPy and SciPy (LAPACK) for the dense linear algebra and Numba for the shifted triangular and Hessenberg solves. The LP $\mathcal{P}(\lambda_0)$, where $\lambda_0$ is the nominal value defined in Section~\ref{sec:exp-instances}, is solved by the
dual simplex method of the solver, and the basis information and the solves with the optimal basis matrix are
obtained from the solver itself, through its C interface (\texttt{CPXgetbhead}, \texttt{CPXftran} and
\texttt{CPXbtran} of CPLEX~22.1.1~\citep{cplex}). The LU factorization of $A_B$ maintained by the simplex method is thus reused, and no
factorization is computed by our code (except for the reference solutions below). The matrix $E$ of the dense
methods is formed with $n_D$ solves with $A_B$, one per nonzero column of $D_B$.

The methods are compared with three baselines: \emph{cold}, which solves $\mathcal{P}(\lambda)$ from scratch for each
$\lambda$ (dual simplex, new model); \emph{solver warm start}, which updates the coefficients $A + \lambda D$ in the model of the
solver and re-optimizes from the previous optimal basis; and \emph{LU}, which computes a sparse LU factorization
(SuperLU) of $A_B + \lambda D_B$ for each $\lambda$, the ``basis system solving'' method of
Table~\ref{tab:complexity}. The LU solutions also serve as references for the accuracy of the methods. All the codes
run on a single thread (solver, BLAS and Numba), on a server with two Intel Xeon E5-2650 v4 processors (2.2~GHz,
24 cores in total) and 128~GB of memory. The server processes 12 instances in parallel.

The instances are general LPs, with bounded variables and both equality and inequality constraints, handled in the form of
Appendix~\ref{sec:lr-bounded}: a slack variable with a column of the identity is added to each constraint, the
nonbasic variables are fixed at the bound given by the solver, and the conditions of Section~\ref{sec:basis_opt}
become $\ell_B \le x_B(\lambda) \le u_B$ and sign conditions on the reduced costs that depend on the bound of each
nonbasic variable. A basis is declared optimal for $\lambda$ when it satisfies the conditions of
Section~\ref{sec:basis_opt} with the accuracy of the certificate given by the solver at $\lambda_0$. The paragraphs below give the
tolerances of these tests, and how the formula of Zuidwijk and the radius of Section~\ref{sec:lr-radius} are evaluated. 

\paragraph{Optimality tolerances.} A basis is declared feasible (resp.\ optimal) for $\lambda$ when no bound (resp.\ no bound and no sign condition
on the reduced costs) is violated, in relative terms, by more than the largest of $10^{-9}$ and its violation at
$\lambda_0$, increased by $10^{-11}$ for the rounding errors of the evaluation: the basis is thus certified at $\lambda$ with the accuracy of the certificate given by the solver at
$\lambda_0$, whose tolerances are $10^{-6}$. (A plain relative tolerance of $10^{-6}$ accepts, on a few degenerate
instances, bases whose reduced costs have become slightly negative and whose objective is then $10^{-5}$ above the
optimum.) 

\paragraph{Singularity and failed decompositions.} $A_B + \lambda D_B$ is considered singular when $\lvert 1 + \lambda\mu\rvert \le 10^{-10}$ for an eigenvalue $\mu$ of $E$, $C$ or $H_j$, depending on the method. For this test, our implementation of the Hessenberg variants computes the eigenvalues of $H$ once, in the preprocessing, rather than using the pivots of the elimination. The eigendecomposition of a matrix is declared
failed (numerically defective) when the condition number $\lVert Q\rVert_1\lVert Q^{-1}\rVert_1$ of its eigenvector
matrix exceeds $10^{15}$. 

\paragraph{Zuidwijk's formula.} Zuidwijk's formula is evaluated as
$o_B(\lambda) = \lambda^{-1}\left(\exp\left(\sum_j\log(1+\lambda\beta_j) - \sum_j\log(1+\lambda\alpha_j)\right) - 1\right)$
to avoid overflows, with the eigenvalues $\alpha_j$ of $E$ and $\beta_j$ of $E + x_B(0)c_B^{\top}$.  

\paragraph{Margins of the radius.} The radius of Section~\ref{sec:lr-radius} uses the two scalings of Section~\ref{sec:radius-choice} and keeps the larger radius for each condition. The margins of the conditions are increased by the tolerances of the certificate above: rounding errors give
derivatives of the order of $10^{-15}$ to conditions that do not move, and with exact margins the radius vanishes
in the direction in which these derivatives happen to be negative (on \radExactMarginsZeroPct{} of the sides of
the stability intervals longer than $10^{-6}$). 

\paragraph{Errors and violations.} Errors are measured relative to $\max(1, \lvert\cdot\rvert)$ for the objective and to $\max(1, \lVert\cdot\rVert_\infty)$ for $x_B$; the violations of a bound and of a sign condition are relative to $1 + \lvert\text{bound}\rvert$ and $1 + \lvert c_j\rvert$, respectively. 

\paragraph{Breakdown of the Arnoldi process.} In the Krylov method of Section~\ref{sec:krylov}, the breakdown test of the Arnoldi process is relative to
$\lVert E v_j\rVert_2$, not to $\lVert x_B(0)\rVert_2$: on problems with large basic solutions, the latter declares
spurious invariant subspaces.

\paragraph{Scalings.} The scalings of Section~\ref{sec:radius-choice} rely on the following facts. A scaling changes $\gamma =
\lVert\Omega^{-1}H\Omega\rVert_\infty$, which is at least the spectral radius $\pi$ of $\lvert H\rvert$, since a matrix
norm is at least the spectral radius and $\Omega^{-1}\lvert H\rvert\Omega$ has the same eigenvalues as $\lvert H\rvert$; it
also changes the vectors $\Omega Q^{\top}\xi$ and $\Omega^{-1}Q^{\top}\zeta$ that define $a$. When $\lvert H\rvert$ is
irreducible, the Perron--Frobenius theorem~\citep[Chapter~8]{HornJohnson.2013} gives a vector $\psi$ with positive entries such
that $\lvert H\rvert \psi = \pi \psi$, and $\Omega = \operatorname{diag}(\psi)$ attains $\gamma = \pi$, since every row of $\Omega^{-1}\lvert H\rvert\Omega$ then sums to $\pi$. We approximate $\psi$ by the power method.

\paragraph{Cost of the radius.} In the computation of Section~\ref{sec:radius-comp}, the coefficients of the basic variables are given by the product $\hat{U}Q$ ($\mathcal{O}(m\rho d)$
operations), and those of the reduced costs by $Q_c^{\top}\theta_j = H_c^{\top}(\hat{V}Q_c)^{\top}A_Ne_j -
(\hat{V}Q_c)^{\top}D_Ne_j$, where $Q_c$ is the basis of the Krylov subspace of $C^{\top}$ and $C^{\top}Q_c = Q_c H_c$
($\mathcal{O}(m\rho d_c + \mathrm{nz}\,d_c + d_c^2\lvert N\rvert)$ operations). The Arnoldi processes require
$\mathcal{O}(\rho^2(d + d_c))$ operations and the radii, once the coefficients are known, $\mathcal{O}(m + \lvert N\rvert)$.

\paragraph{Conditions that do not move.} A condition with a zero margin and a zero derivative gets a zero radius if its coefficient $a$ is nonzero. This is a limitation of a first-order bound, since such a condition may remain satisfied at second order; when the condition does not move at all, however, all the coefficients $\xi^{\top}K^k\zeta$ of its remainder vanish, so that $\xi$ is orthogonal to the Krylov subspace of Lemma~\ref{lem:reduced-bound} and $a$ is zero.

\section{Instances}\label{app:instances}
The instances come from the dataset of \citet{dataset} and are summarized in Table~\ref{tab:exp-dataset}. All of
them define a matrix $D$ and a range $[\underline\lambda, \overline\lambda]$ for $\lambda$:
\begin{itemize}
    \item \textbf{Netlib}: the LPs of the Netlib collection~\citep{netlib} that are solved within a few seconds; 100
    randomly chosen equality constraints (instances \texttt{\_eq}) or inequality constraints (\texttt{\_ineq}) are
    perturbed, by changing up to three of their coefficients $a_{ij}$ into $a_{ij}(1 \pm \lambda u_{ij})$ with
    $u_{ij}$ uniform in $[0.3, 0.5]$, for $\lambda\in[-1, 1]$;
    \item \textbf{MIPLIB}: instances of MIPLIB~2017~\citep{Gleixner.2021}, perturbed in the same way;
    \item \textbf{Facility}: the capacitated facility location instances of \citet{Holmberg.1999}, with all
    capacities multiplied by $\lambda\in[0.7, 1]$;
    \item \textbf{Unit commitment}: the unit commitment instances used by \citet{Frangioni.2009}, in which the coefficients
    of the power outputs of the units in the hourly load-balance constraints are perturbed by random relative amounts
    of at most $\pm 30\%$, for $\lambda \in [-1, 1]$;
    \item \textbf{Real models}: 126 (model, sensitivity) pairs built from public models of hub
    location~\citep{Ernst.1999}, index tracking~\citep{Beasley.2003}, unit commitment~\citep{Knueven.2020},
    network-constrained unit commitment~\citep{Brown.2018}, telecommunication network design~\citep{Orlowski.2010},
    retail facility location~\citep{Feng.2022} and batch process scheduling~\citep{Kondili.1993}, where $D$ describes a
    physically meaningful change of the model (a capacity transferred between hubs or reactors, a fuel price, a
    reaction selectivity, a line reactance, \dots), for $\lambda\in[-1,1]$ (or $[-0.9, 0.9]$).
\end{itemize}
Except for Netlib, the models are mixed-integer, and we consider their LP relaxations. The methods of this paper apply to these relaxations, and branch and bound solves such a relaxation at each node. Table~\ref{tab:exp-dataset} also reports $n_D$, the number of nonzero columns of $D_B$ for the optimal
basis $B$ of $\mathcal{P}(\lambda_0)$, and the dimension $\rho$ of the low-rank matrix $C$ of
Section~\ref{sec:lowrank}.

\begin{table}
\TABLE
{Instances.\label{tab:exp-dataset}}
{\small
    \input{tables/tab_dataset.tex}}
{Number of instances, number of constraints $m$ and of variables $n$ (after the variable bounds
    are restored as bounds), nonzeros of $D$, number $n_D$ of perturbed basic columns and dimension $\rho$ of the
    low-rank matrix $C$ for the optimal basis at $\lambda_0$.}
\end{table}

\section{Proofs}\label{app:proofs}
\begin{proof}{Proof of Lemma~\ref{lem:lr-spectrum}.}
The first claim is a classical result on the products $\hat{U}V^{\top}$ and $V^{\top}\hat{U}$ of two matrices $\hat{U}, V \in
    \mathbb{R}^{m\times \rho}$~\citep{Flanders.1951}. The second claim follows from Theorem~\ref{thm:lr} and $\det(I_\rho +
    \lambda C) = \prod_{\mu} (1 + \lambda \mu)$, the product running over the eigenvalues of $C$ counted with
    multiplicities.\Halmos
\end{proof}

\begin{proof}{Proof of Lemma~\ref{lem:lr-diag}.}
A matrix is diagonalizable if and only if each of its eigenvalues has equal algebraic and geometric
    multiplicities. By Lemma~\ref{lem:lr-spectrum}, this holds for the nonzero eigenvalues of $E$ if and only if it
    holds for those of $C$. With a minimal factorization, $\hat{U}, V \in \mathbb{R}^{m\times \rho}$ have full column rank $\rho$.
    As $E x = \hat{U}(V^{\top} x)$ and the columns of $\hat{U}$ are linearly independent, $E x = 0$ if and only if $V^{\top} x = 0$, so
    that the geometric multiplicity of the eigenvalue $0$ of $E$ is $\dim\ker(E) = \dim\ker(V^{\top}) = m - \rho$. Its
    algebraic multiplicity is $m - \rho$ plus that of $0$ in $C$ (the algebraic multiplicities add up to $m$ in $E$
    and to $\rho$ in $C$, and coincide for the nonzero eigenvalues by Lemma~\ref{lem:lr-spectrum}). The two
    multiplicities of $0$ in $E$ are thus equal if and only if $C$ is nonsingular, and then all the eigenvalues of
    $C$ are nonzero.\Halmos
\end{proof}

\begin{proof}{Proof of Lemma~\ref{lem:first-order}.}
Section~\ref{sec:first-order} derives $x_B(\lambda) = x_B(0) + \lambda\dot x_B - \lambda^2 EW\dot x_B$, with $W = (I_m + \lambda E)^{-1}$. The identity $(I_m + \lambda \hat{U}V^{\top})^{-1}\hat{U} = \hat{U}(I_\rho + \lambda C)^{-1}$ gives $EW = \hat{U}(I_\rho + \lambda C)^{-1}V^{\top}$, hence~\eqref{eq:rem-x}, and multiplying it by $c_B^{\top}$ gives~\eqref{eq:rem-o}. For the reduced costs, $A_B + \lambda D_B = A_B(I_m + \lambda E)$ gives $r(\lambda) = c_N^{\top} - c_B^{\top}WA_B^{-1}(A_N + \lambda D_N)$. Replacing $W$ by $I_m - \lambda WE$ where it multiplies $A_B^{-1}A_N$ gives $r(\lambda) = r(0) + \lambda c_B^{\top}WG$, and replacing $W$ by $I_m - \lambda EW$ then gives $r(\lambda) = r(0) + \lambda\dot r - \lambda^2c_B^{\top}EWG$. With $c_B^{\top}\hat{U} = w_c^{\top}$, the last term is $-\lambda^2 w_c^{\top}(I_\rho + \lambda C)^{-1}V^{\top}Ge_j = -\lambda^2\theta_j^{\top}(I_\rho + \lambda C^{\top})^{-1}w_c$ for each $j \in N$.\Halmos
\end{proof}

\begin{proof}{Proof of Lemma~\ref{lem:reduced-bound}.}
The Krylov subspace $\mathcal{K} = \mathcal{K}_\rho(K,\zeta)$ contains $\zeta$ and is invariant under $K$: by the
    Cayley--Hamilton theorem, $K^\rho\zeta$ is a linear combination of $\zeta, K\zeta, \dots, K^{\rho-1}\zeta$. The
    matrix $QQ^{\top}$ is the orthogonal projector onto $\mathcal{K}$: it leaves the vectors of
    $\mathcal{K}$ unchanged, but it is not $I_\rho$ when $d < \rho$. As $\mathcal{K}$ is invariant under $K$, the columns
    of $K Q$ belong to $\mathcal{K}$, so that $K Q = QQ^{\top}K Q =
    QH$ and $(I_\rho + \lambda K)Q = Q(I_d + \lambda H)$, where $I_d +
    \lambda H$ is invertible since $\lvert\lambda\rvert\lVert\Omega^{-1}H\Omega\rVert_\infty < 1$. As
    $\zeta \in \mathcal{K}$, $\zeta = QQ^{\top}\zeta$, so that $(I_\rho + \lambda K)^{-1}\zeta =
    Q(I_d + \lambda H)^{-1}Q^{\top}\zeta$, and
    \begin{equation*}
        \xi^{\top}(I_\rho + \lambda K)^{-1}\zeta = (\Omega Q^{\top}\xi)^{\top}\,(I_d + \lambda\,\Omega^{-1}H
        \Omega)^{-1}\,(\Omega^{-1}Q^{\top}\zeta).
    \end{equation*}
    The bound follows from H\"older's inequality and the Neumann bound of Section~\ref{sec:remainders}.\Halmos
\end{proof}

\begin{proof}{Proof of Lemma~\ref{lem:first-root}.}
The function $t \mapsto t^2/(1 - \gamma t)$ is convex on $[0, 1/\gamma)$, so that $f$ is concave there and, as $f(0) = \delta \ge 0$, the set of the $t$ such that $f(t) \ge 0$ is an interval that contains $0$. In the first case, $f(t) < 0$ for all $t > 0$, and in the second one $f$ is identically zero. Otherwise, $f$ is positive on some interval $(0, \varepsilon)$, and the interval ends at the first root of $f$ in $(0, 1/\gamma)$, or at $1/\gamma$ if there is none. As $1 - \gamma t > 0$ on $[0, 1/\gamma)$, these roots are those of $q(t) = (1 - \gamma t)f(t)$.\Halmos
\end{proof}

\end{APPENDICES}

\end{document}

%% file: tables/numbers.tex
\providecommand{\nInst}{}\renewcommand{\nInst}{398}
\providecommand{\nInstSmall}{}\renewcommand{\nInstSmall}{310}
\providecommand{\denseMax}{}\renewcommand{\denseMax}{4985}
\providecommand{\mMax}{}\renewcommand{\mMax}{349\,120}

\providecommand{\rMedian}{}\renewcommand{\rMedian}{29}
\providecommand{\rMax}{}\renewcommand{\rMax}{669}
\providecommand{\pGrid}{}\renewcommand{\pGrid}{101}

\providecommand{\pLocal}{}\renewcommand{\pLocal}{1001}
\providecommand{\eigFail}{}\renewcommand{\eigFail}{214}
\providecommand{\eigFailPct}{}\renewcommand{\eigFailPct}{69\%}

\providecommand{\lrEigFail}{}\renewcommand{\lrEigFail}{128}
\providecommand{\accMaxObj}{}\renewcommand{\accMaxObj}{$1.9\cdot 10^{-10}$}
\providecommand{\accMaxX}{}\renewcommand{\accMaxX}{$1.5\cdot 10^{-9}$}

\providecommand{\zuidMaxErr}{}\renewcommand{\zuidMaxErr}{$3.8\cdot 10^{-9}$}

\providecommand{\objLrHessMax}{}\renewcommand{\objLrHessMax}{1.4 ms}
\providecommand{\objKrylovMax}{}\renewcommand{\objKrylovMax}{133 $\mu$s}

\providecommand{\certSpeedupLrHess}{}\renewcommand{\certSpeedupLrHess}{8.9$\times$}
\providecommand{\certSpeedupKrylov}{}\renewcommand{\certSpeedupKrylov}{9.1$\times$}

\providecommand{\krylovItersMedian}{}\renewcommand{\krylovItersMedian}{8}
\providecommand{\krylovItersMax}{}\renewcommand{\krylovItersMax}{38}

\providecommand{\intRadiusMedian}{}\renewcommand{\intRadiusMedian}{$7.3\cdot 10^{-4}$}
\providecommand{\intWidthMedian}{}\renewcommand{\intWidthMedian}{$8.6\cdot 10^{-3}$}

\providecommand{\coverMedian}{}\renewcommand{\coverMedian}{1}
\providecommand{\degPrimalPct}{}\renewcommand{\degPrimalPct}{85\%}

\providecommand{\intOneSidedPct}{}\renewcommand{\intOneSidedPct}{10\%}
\providecommand{\intZeroSidedPct}{}\renewcommand{\intZeroSidedPct}{20\%}
\providecommand{\curveMaxErr}{}\renewcommand{\curveMaxErr}{$1.4\cdot 10^{-4}$}
\providecommand{\warmMaxErr}{}\renewcommand{\warmMaxErr}{$1.4\cdot 10^{-4}$}
\providecommand{\warmMaxErrInstance}{}\renewcommand{\warmMaxErrInstance}{miplib/sp150x300d\_eq}
\providecommand{\curveMaxErrSolver}{}\renewcommand{\curveMaxErrSolver}{$7.4\cdot 10^{-8}$}
\providecommand{\curveInstance}{}\renewcommand{\curveInstance}{real/sndlib\_abilene\_\_east\_west\_shift}
\providecommand{\basesPerPoint}{}\renewcommand{\basesPerPoint}{0.52}

\providecommand{\slowerSchurLarge}{}\renewcommand{\slowerSchurLarge}{97\%}

\providecommand{\localSchurFasterMedianM}{}\renewcommand{\localSchurFasterMedianM}{162}
\providecommand{\localSchurSlowerMedianM}{}\renewcommand{\localSchurSlowerMedianM}{504}

\providecommand{\netlibBasesPerPoint}{}\renewcommand{\netlibBasesPerPoint}{0.48}
\providecommand{\localSpeedupKrylov}{}\renewcommand{\localSpeedupKrylov}{5.5$\times$}
\providecommand{\localSpeedupLrHess}{}\renewcommand{\localSpeedupLrHess}{3.7$\times$}

\providecommand{\radNSides}{}\renewcommand{\radNSides}{595}
\providecommand{\radNTrivial}{}\renewcommand{\radNTrivial}{201}

\providecommand{\radRatioMedian}{}\renewcommand{\radRatioMedian}{0.9996}
\providecommand{\radHalfPct}{}\renewcommand{\radHalfPct}{90\%}
\providecommand{\radTenthPct}{}\renewcommand{\radTenthPct}{95\%}

\providecommand{\radTimeMedian}{}\renewcommand{\radTimeMedian}{14 ms}
\providecommand{\radExactTimeMedian}{}\renewcommand{\radExactTimeMedian}{119 ms}

\providecommand{\radTimeSpeedup}{}\renewcommand{\radTimeSpeedup}{13$\times$}

\providecommand{\radExactMarginsZeroPct}{}\renewcommand{\radExactMarginsZeroPct}{32\%}
\providecommand{\radNEndSides}{}\renewcommand{\radNEndSides}{80}

\providecommand{\radNoRedHalfPct}{}\renewcommand{\radNoRedHalfPct}{59\%}

\providecommand{\radDenseNSides}{}\renewcommand{\radDenseNSides}{458}

\providecommand{\radDenseHalfPct}{}\renewcommand{\radDenseHalfPct}{29\%}

\providecommand{\radDimPrimalMedian}{}\renewcommand{\radDimPrimalMedian}{0.12}
\providecommand{\radDimDualMedian}{}\renewcommand{\radDimDualMedian}{0.29}

\providecommand{\gridNInst}{}\renewcommand{\gridNInst}{392}
\providecommand{\gridPSmall}{}\renewcommand{\gridPSmall}{101}

\providecommand{\gridDistinctSmall}{}\renewcommand{\gridDistinctSmall}{52}

\providecommand{\gridRatioKrylovSmall}{}\renewcommand{\gridRatioKrylovSmall}{1.70}
\providecommand{\gridWinPctKrylovSmall}{}\renewcommand{\gridWinPctKrylovSmall}{18\%}

\providecommand{\gridPMid}{}\renewcommand{\gridPMid}{1001}

\providecommand{\gridDistinctMid}{}\renewcommand{\gridDistinctMid}{99}

\providecommand{\gridRatioKrylovMid}{}\renewcommand{\gridRatioKrylovMid}{1.00}
\providecommand{\gridWinPctKrylovMid}{}\renewcommand{\gridWinPctKrylovMid}{50\%}

\providecommand{\gridPLarge}{}\renewcommand{\gridPLarge}{10001}

\providecommand{\gridDistinctLarge}{}\renewcommand{\gridDistinctLarge}{130}

\providecommand{\gridRatioKrylovLarge}{}\renewcommand{\gridRatioKrylovLarge}{0.41}
\providecommand{\gridWinPctKrylovLarge}{}\renewcommand{\gridWinPctKrylovLarge}{80\%}
\providecommand{\gridSpeedupMaxKrylovLarge}{}\renewcommand{\gridSpeedupMaxKrylovLarge}{34$\times$}

\providecommand{\gridGrowthWarm}{}\renewcommand{\gridGrowthWarm}{52$\times$}
\providecommand{\gridGrowthKrylov}{}\renewcommand{\gridGrowthKrylov}{7.3$\times$}
\providecommand{\gridWinFewBases}{}\renewcommand{\gridWinFewBases}{95\%}
\providecommand{\gridWinManyBases}{}\renewcommand{\gridWinManyBases}{0\%}
\providecommand{\gridMaxErr}{}\renewcommand{\gridMaxErr}{$3.9\cdot 10^{-7}$}

\providecommand{\localNocertSpeedupLrHess}{}\renewcommand{\localNocertSpeedupLrHess}{117$\times$}

\providecommand{\localNocertNLrEig}{}\renewcommand{\localNocertNLrEig}{270}
\providecommand{\localNocertSpeedupKrylov}{}\renewcommand{\localNocertSpeedupKrylov}{304$\times$}

\providecommand{\localNocertRadiusSpeedupLrHess}{}\renewcommand{\localNocertRadiusSpeedupLrHess}{34$\times$}
\providecommand{\localNocertExactSpeedupLrHess}{}\renewcommand{\localNocertExactSpeedupLrHess}{6.1$\times$}
\providecommand{\localNocertRadiusSpeedupKrylov}{}\renewcommand{\localNocertRadiusSpeedupKrylov}{39$\times$}
\providecommand{\localNocertExactSpeedupKrylov}{}\renewcommand{\localNocertExactSpeedupKrylov}{6.4$\times$}

%% file: tables/tab_timing.tex
\setlength{\tabcolsep}{4pt}\footnotesize
\setlength{\tabcolsep}{3pt}
\begin{tabular}{lrrrrrrrr}
\toprule
 & \multicolumn{4}{c}{310 instances with $m \le 4985$} & \multicolumn{4}{c}{All the 398 instances}\\
\cmidrule(lr){2-5}\cmidrule(lr){6-9}
Method & Prep. & Obj. / $\lambda$ & Cert. / $\lambda$ & Speed-up & Prep. & Obj. / $\lambda$ & Cert. / $\lambda$ & Speed-up\\
\midrule
Eigendecomposition & 3.5 ms & 3.2 $\mu$s & 74 $\mu$s & 3.2$\times$ & -- & -- & -- & --\\
Schur & 18 ms & 297 $\mu$s & 1.1 ms & 0.7$\times$ & -- & -- & -- & --\\
Hessenberg (dense) & 12 ms & 447 $\mu$s & 1.8 ms & 0.4$\times$ & -- & -- & -- & --\\
Zuidwijk & 20 ms & 64 $\mu$s & -- & -- & -- & -- & -- & --\\
Low-rank + eig. & 2.9 ms & 0.38 $\mu$s & 74 $\mu$s & 7.0$\times$ & 4.8 ms & 0.41 $\mu$s & 159 $\mu$s & 9.9$\times$\\
Low-rank + Schur & 3.9 ms & 2.7 $\mu$s & 136 $\mu$s & 6.5$\times$ & 5.7 ms & 2.3 $\mu$s & 190 $\mu$s & 7.9$\times$\\
Low-rank + Hessenberg & 4.1 ms & 3.4 $\mu$s & 133 $\mu$s & 6.6$\times$ & 6.2 ms & 3 $\mu$s & 180 $\mu$s & 8.9$\times$\\
Shifted Krylov & 1.2 ms & 1.9 $\mu$s & 128 $\mu$s & 7.4$\times$ & 1.3 ms & 2.7 $\mu$s & 172 $\mu$s & 9.1$\times$\\
\midrule
LP solve from scratch & -- & \multicolumn{2}{r}{10 ms} & & -- & \multicolumn{2}{r}{25 ms} & \\
Solver warm start & -- & \multicolumn{2}{r}{602 $\mu$s} & & -- & \multicolumn{2}{r}{1.1 ms} & \\
\bottomrule
\end{tabular}

%% file: tables/tab_radius.tex
\setlength{\tabcolsep}{4pt}\footnotesize
\begin{tabular}{lrrrrrrrr}
\toprule
Family & \# & Sides & \multicolumn{3}{c}{$\Delta^\sigma / \Delta^{\sigma}_{\mathrm{ex}}$} & \multicolumn{3}{c}{$\Delta^\sigma \ge \Delta^{\sigma}_{\mathrm{ex}}/2$}\\
\cmidrule(lr){4-6}\cmidrule(lr){7-9}
 & & & median & $\ge 1/2$ & $\ge 1/10$ & identity & no red. & norms $E$\\
\midrule
Netlib & 115 & 101 & 0.79 & 58\% & 78\% & 46\% & 16\% & 6\%\\
MIPLIB & 31 & 28 & 0.87 & 64\% & 79\% & 61\% & 39\% & 30\%\\
Facility & 72 & 143 & 1.00 & 100\% & 100\% & 100\% & 62\% & 37\%\\
Unit comm. & 54 & 92 & 1.00 & 100\% & 100\% & 100\% & 100\% & 96\%\\
Real models & 126 & 231 & 1.00 & 96\% & 99\% & 89\% & 61\% & 25\%\\
\midrule
All & 398 & 595 & 1.00 & 90\% & 95\% & 85\% & 59\% & 29\%\\
\bottomrule
\end{tabular}

%% file: tables/tab_curve.tex
\setlength{\tabcolsep}{4pt}\footnotesize
\begin{tabular}{lrrrrrrrr}
\toprule
Family & \# & Bases & \multicolumn{2}{c}{Coverage of $B_0$} & \multicolumn{4}{c}{Total time / cold re-solves}\\
 & & / point & \# points & radius & solver warm start & LR+Hess. & Krylov & LU\\
\midrule
Netlib & 115 & 0.48 & 1 & $1.0\cdot 10^{-9}$ & 0.08 & 0.30 & 0.19 & 0.59\\
MIPLIB & 31 & 0.12 & 1 & $2.2\cdot 10^{-9}$ & 0.05 & 0.07 & 0.05 & 0.35\\
Facility & 72 & 0.07 & 26 & $1.3\cdot 10^{-2}$ & 0.08 & 0.09 & 0.09 & 0.42\\
Unit comm. & 54 & 1.00 & 1 & $4.1\cdot 10^{-5}$ & 0.09 & 0.16 & 0.16 & 0.31\\
Real models & 126 & 0.80 & 1 & $2.2\cdot 10^{-3}$ & 0.06 & 0.13 & 0.12 & 0.48\\
\midrule
All & 398 & 0.52 & 1 & $7.3\cdot 10^{-4}$ & 0.07 & 0.13 & 0.12 & 0.44\\
\bottomrule
\end{tabular}

%% file: tables/tab_grid.tex
\setlength{\tabcolsep}{4pt}\footnotesize
\setlength{\tabcolsep}{2.5pt}
\begin{tabular}{lrrrrrrrrrr}
\toprule
Family & \# & $p$ & Basis & Distinct & \multicolumn{3}{c}{Total time (median)} & \multicolumn{2}{c}{\begin{tabular}[b]{@{}c@{}}Reuse / solver\\warm start\end{tabular}} & Faster\\
\cmidrule(lr){6-8}\cmidrule(lr){9-10}
 & & & changes & bases & \begin{tabular}[b]{@{}r@{}}Solver\\warm start\end{tabular} & LR+Hess. & Krylov & LR+Hess. & Krylov & (Krylov)\\
\midrule
Netlib & 115 & 101 & 48 & 47 & 99 ms & 553 ms & 392 ms & 3.75 & 2.53 & 13\%\\
 &  & 1001 & 85 & 72 & 653 ms & 1.4 s & 787 ms & 1.69 & 0.99 & 51\%\\
 &  & 10001 & 106 & 86 & 6 s & 3.4 s & 2.1 s & 0.51 & 0.28 & 91\%\\
\addlinespace
MIPLIB & 30 & 101 & 13 & 13 & 350 ms & 704 ms & 479 ms & 1.08 & 1.00 & 50\%\\
 &  & 1001 & 20 & 14 & 2.2 s & 2 s & 1.3 s & 0.71 & 0.52 & 73\%\\
 &  & 10001 & 25 & 14 & 19 s & 9.6 s & 7.6 s & 0.43 & 0.35 & 93\%\\
\addlinespace
Facility & 72 & 101 & 7 & 7 & 36 ms & 53 ms & 57 ms & 1.29 & 1.25 & 40\%\\
 &  & 1001 & 8 & 8 & 283 ms & 121 ms & 132 ms & 0.39 & 0.43 & 99\%\\
 &  & 10001 & 8 & 8 & 2.8 s & 710 ms & 716 ms & 0.26 & 0.27 & 100\%\\
\addlinespace
Unit comm. & 54 & 101 & 101 & 101 & 12 s & 17 s & 17 s & 1.44 & 1.36 & 0\%\\
 &  & 1001 & 1000 & 1000 & 37 s & 101 s & 82 s & 2.68 & 2.20 & 0\%\\
 &  & 10001 & 5669 & 5669 & 193 s & 463 s & 352 s & 2.62 & 1.96 & 4\%\\
\addlinespace
Real models & 121 & 101 & 78 & 78 & 274 ms & 823 ms & 810 ms & 2.29 & 2.05 & 9\%\\
 &  & 1001 & 230 & 228 & 1.3 s & 2.6 s & 2.3 s & 1.57 & 1.51 & 36\%\\
 &  & 10001 & 326 & 302 & 11 s & 6.9 s & 6.3 s & 0.62 & 0.49 & 87\%\\
\midrule
All & 392 & 101 & 52 & 52 & 126 ms & 563 ms & 486 ms & 1.81 & 1.70 & 18\%\\
 &  & 1001 & 112 & 99 & 763 ms & 1.6 s & 1.2 s & 1.45 & 1.00 & 50\%\\
 &  & 10001 & 155 & 130 & 6.6 s & 4.9 s & 3.5 s & 0.50 & 0.41 & 80\%\\
\bottomrule
\end{tabular}

%% file: tables/tab_local.tex
\setlength{\tabcolsep}{4pt}\footnotesize
\setlength{\tabcolsep}{3pt}
\begin{tabular}{lrrrrrrrrrr}
\toprule
 & & & \multicolumn{4}{c}{With the optimality test} & \multicolumn{4}{c}{Without the optimality test}\\
\cmidrule(lr){4-7}\cmidrule(lr){8-11}
Family & \# & \begin{tabular}[b]{@{}r@{}}Solver warm\\start / $\lambda$\end{tabular} & LR+Hess. & LR+eig. & Krylov & LU & LR+Hess. & LR+eig. & Krylov & LU\\
\midrule
Netlib & 115 & 579 $\mu$s & 3.8$\times$ & 5.2$\times$ & 7.6$\times$ & 0.1$\times$ & 24.7$\times$ & 125.0$\times$ & 256.9$\times$ & 0.3$\times$\\
MIPLIB & 31 & 1.8 ms & 3.0$\times$ & 3.8$\times$ & 4.5$\times$ & 0.1$\times$ & 242.3$\times$ & 506.6$\times$ & 819.2$\times$ & 0.5$\times$\\
Facility & 72 & 287 $\mu$s & 3.5$\times$ & 4.0$\times$ & 4.1$\times$ & 0.2$\times$ & 114.9$\times$ & 211.4$\times$ & 186.7$\times$ & 0.7$\times$\\
Unit comm. & 54 & 9.7 ms & 3.6$\times$ & 3.2$\times$ & 4.9$\times$ & 0.1$\times$ & 508.5$\times$ & 561.6$\times$ & 961.3$\times$ & 0.2$\times$\\
Real models & 126 & 1.1 ms & 3.9$\times$ & 3.8$\times$ & 6.5$\times$ & 0.1$\times$ & 319.4$\times$ & 200.4$\times$ & 395.0$\times$ & 0.3$\times$\\
\midrule
All & 398 & 650 $\mu$s & 3.7$\times$ & 3.9$\times$ & 5.5$\times$ & 0.1$\times$ & 117.3$\times$ & 238.1$\times$ & 303.9$\times$ & 0.3$\times$\\
\bottomrule
\end{tabular}

%% file: tables/tab_dataset.tex
\setlength{\tabcolsep}{4pt}\footnotesize
\begin{tabular}{lrrrrrr}
\toprule
Family & \# & $m$ & $n$ & nnz$(D)$ & $n_D$ & $\rho$\\
 & & med (min--max) & med (min--max) & med & med (max) & med (max)\\
\midrule
Netlib & 115 & 551 (123--3100) & 1371 (83--13525) & 148 & 91 (153) & 83 (100)\\
MIPLIB & 31 & 4435 (450--149855) & 10100 (600--111827) & 141 & 60 (135) & 58 (98)\\
Facility & 72 & 115 (60--230) & 2075 (510--6030) & 20 & 20 (30) & 20 (30)\\
Unit comm. & 54 & 32792 (3247--73792) & 7500 (750--22500) & 2400 & 2391 (4793) & 24 (24)\\
Real models & 126 & 1275 (144--349120) & 3296 (270--305712) & 43 & 21 (669) & 21 (669)\\
\midrule
All & 398 & 712 (60--349120) & 2480 (83--305712) & 118 & 48 (4793) & 29 (669)\\
\bottomrule
\end{tabular}